\documentclass[11pt]{amsart} \textwidth=14.5cm \oddsidemargin=1cm
\usepackage{setspace}
\usepackage{amsthm}
\usepackage[rgb]{xcolor}
\usepackage{times}
\usepackage[linktocpage=true]{hyperref}
\usepackage{mathrsfs}  

\usepackage{xypic}
\usepackage{tikz}
\usepackage{verbatim}
\usepackage{amscd}
\usepackage{amsmath}
\usepackage{amsfonts}
\usepackage{amsxtra}
\usepackage{amssymb}
\usepackage{eucal}
\usepackage{latexsym,todonotes}
\usepackage{stmaryrd}
\usepackage{graphicx}%
\usepackage{dynkin-diagrams}
\usepackage{tikz}
\usetikzlibrary{intersections}
\newtheorem{theorem}{Theorem}[section]

\newtheorem{corollary}[theorem]{Corollary}

\newtheorem{definition}[theorem]{Definition}
\newtheorem{example}[theorem]{Example}

\newtheorem{lemma}[theorem]{Lemma}

\newtheorem{proposition}[theorem]{Proposition}

\numberwithin{equation}{section}

\theoremstyle{remark}
\newtheorem{remark}[theorem]{Remark}

\newtheorem{innercustomthm}{{\bf Theorem}}
\newenvironment{customthm}[1]
  {\renewcommand\theinnercustomthm{#1}\innercustomthm}
  {\endinnercustomthm}

\newcommand{\cM}{\mathcal{M}}
\newcommand{\N}{\mathbb{N}}
\newcommand{\Z}{\mathbb{Z}}

\newcommand{\Q}{\mathbb{Q}}
\newcommand{\C}{\mathbb{C}}

\newcommand{\K}{\mathcal K}
\newcommand{\Hy}{\mathscr{H}}

\newcommand{\HB}{\mathscr H_{B_d}}

\newcommand{\End}{\text{End}\ }

\newcommand{\bD}{\overline{\Delta}}
\newcommand{\D}{\DH}
\newcommand{\inv}{\text{inv}}
\newcommand{\M}{\mathbb M}
\newcommand{\V}{\mathbb V}
\newcommand{\la}{\langle}
\newcommand{\ra}{\rangle}

\newcommand{\io}{\imath}
\newcommand{\bu}{\bullet}
\newcommand{\U}{\mathbf U}
\newcommand{\Ui}{\mathbf U^\imath}
\newcommand{\va}{\varsigma}

\newcommand{\sch}{\mathcal S}
\newcommand{\iba}{\psi_{\io}}
\newcommand{\T}{T}
\newcommand{\pt}{n}
\newcommand{\nb}{m}
\newcommand{\up}{\Upsilon}
\newcommand{\I}{\mathbb I}
\newcommand{\Ibw}{\I_{r|m|r}}
\newcommand{\Ianti}{\I_{r|m|r}^{d,-}}
\newcommand{\Iwl}{\mathbb I_{\circ}^{-}}
\newcommand{\Iwr}{\mathbb I_{\circ}^{+}}
\newcommand{\Ib}{\mathbb I_{\bu}}
\newcommand{\qbinom}[2]{\begin{bmatrix} #1\\#2 \end{bmatrix} }

\begin{document}
\title{$\imath$Schur duality and Kazhdan-Lusztig basis expanded}
\author{Yaolong Shen and Weiqiang Wang}
\address{Department of Mathematics, University of Virginia, Charlottesville, VA 22904}
\email{ys8pfr@virginia.edu (Shen), ww9c@virginia.edu (Wang)}


\begin{abstract}
Expanding the classical works of Kazhdan-Lusztig and Deodhar, we establish bar involutions and canonical (i.e., quasi-parabolic KL) bases on quasi-permutation modules over the type B Hecke algebra, where the bases are parameterized by cosets of (possibly non-parabolic) reflection subgroups of the Weyl group of type B. We formulate an $\imath$Schur duality between an $\imath$quantum group of type AIII (allowing black nodes in its Satake diagram) and a Hecke algebra of type B acting on a tensor space, providing a common generalization of Jimbo-Schur duality and Bao-Wang's quasi-split $\imath$Schur duality. The quasi-parabolic KL bases on quasi-permutation Hecke modules are shown to match with the $\imath$canonical basis on the tensor space. An inversion formula for quasi-parabolic KL polynomials is established via the $\imath$Schur duality. 
\end{abstract}

\maketitle
 \setcounter{tocdepth}{1}

\tableofcontents

\section{Introduction}

\subsection{Type B Kazhdan-Lusztig, expanded}

Let $W =W_d$ be the Weyl group of type $B_d$ generated by the simple reflections $s_0, s_1, \ldots, s_{d-1}$, which contains the symmetric group $S_d$ naturally as a subgroup. Let $\HB$ be its associated Hecke algebra generated by $H_0, H_1, \ldots, H_{d-1}$ in 2 parameters $q,p$, which contains the Hecke algebra $\Hy_{S_d}$ as a subalgebra. (In the introduction, we shall assume that $p$ is an integer power of $q$; a reader can take $p=q$.)

Consider reflection subgroups of $W_d$ of the form 
\begin{align}
  \label{eq:Wf1}
W_f =W_{{m_1}}\times \ldots \times W_{{m_k}}\times S_{m_{k+1}}\times \ldots\times S_{m_l}. \end{align} 
where $m_1+\ldots+m_l=d$, $k\le l$ and all $m_i$ are positive. Clearly, $W_f$ is a parabolic subgroup of $W_d$ if and only if $k \le 1$.  For $k\le 1$,  there exists a right $\HB$-module $\M_f$, the induced module from the trivial module of the subalgebra $\Hy_{W_f}$, parameterized by the set ${}^fW$ of right minimal length representatives of $W_f$. The celebrated Kazhdan-Lusztig (KL) basis on the regular representation of $\HB$ (see \cite{KL79} for $p=q$, and \cite{Lus03} for $p\in q^\Z$) admits a parabolic generalization in terms of $\M_f$ (see Deodhar  \cite{De87}); that is, $\M_f$ admits a bar involution and a distinguished bar-invariant basis, known as the parabolic KL basis.

Our first main result is to extend the above classical works of Kazhdan, Lusztig and Deodhar to construct canonical bases (also called quasi-parabolic KL bases)  of type B associated to arbitrary reflection subgroups $W_f$ of the form \eqref{eq:Wf1}.  
By definition, our modules $\M_f$ depend only on the reflection subgroup $W_f$ of $W_d$, and each $\M_f$ comes with a standard basis $\{M_{f\cdot \sigma} \}$, where $\sigma$ runs over the set ${}^fW$ of minimal length representatives of right cosets of $W_f$ in $W_d$. We denote by $<$ the Chevalley-Bruhat order on ${}^fW$. 

\begin{customthm} {\bf A}
[Proposition~\ref{prop:iHba}, Theorem~\ref{thm:CBMf}]
\label{thm:A}
(1) There exists an anti-linear bar involution $\iba$ on $\M_f$ such that $\iba(M_f) =M_f$, which is compatible with the bar operator on $\HB$, i.e., 
$\iba(x h)=\iba(x) \bar{h},$ for all $x \in \M_f, \  h \in \HB$.

(2) The module $\M_f$ admits a canonical basis $\{C_{\sigma} | \sigma \in {}^fW \}$ such that $C_{ \sigma}$ is bar invariant and  
$C_{\sigma} \in  M_{f\cdot \sigma}+\sum_{w\in {}^fW, w<\sigma} 
q^{-1}\Z[q^{-1}] M_{f\cdot w}.$
\end{customthm}

The module $\M_f$ admits a dual canonical basis $\{C^*_{\sigma} | \sigma \in {}^fW \}$ such that $C^*_{ \sigma}$ is bar invariant and  
$C^*_{\sigma} \in  M_{f\cdot \sigma}+\sum_{w\in {}^fW, w<\sigma} 
q \Z[q] M_{f\cdot w}$; see Proposition~\ref{prop:dualCBMf}. 

Theorem~\ref{thm:A} is totally unexpected when $W_f$ is not parabolic, given the fundamental importance of Kazhdan-Lusztig bases and how well they have been studied from various viewpoints since 1970's. We are led to the formulation of this result from a new $\imath$Schur duality and the corresponding $\imath$canonical bases, which we shall explain below momentarily. 

As $W_f$ may not be parabolic, the Hecke algebra $\Hy(W_f)$ is not a subalgebra of $\HB$ in any natural manner, and hence 
$\M_f$ is not an induced module from an $\Hy(W_f)$-module in general. Accordingly, it is more difficult to establish a key property (see Theorem~\ref{thm:Wf}) concerning the action of the simple reflections $s_i$ on the poset ${}^fW$, generalizing the parabolic case in \cite{De77, De87}. This leads to explicit formulas (see Proposition~\ref{prop:HMf}) for the actions of the generators $H_i$ of $\HB$ on the standard basis of $\M_f$ parametrized by the minimal length coset representatives for $W_f \backslash W$; remarkably, these formulas look identical to those for $W_f$ parabolic. The self-contained proof of Theorem~\ref{thm:A} (which is independent of $\imath$Schur duality below)  will occupy Section~\ref{sec:Mf}. 

The canonical bases in Theorem~\ref{thm:A} include parabolic KL bases of type A (besides those of type B) as special cases. For example, consider the non-parabolic subgroup $W_f =W_1\times \ldots \times W_1$ (generated by the $d$ sign reflections). In this case, ${}^fW =S_d$, and the canonical basis of $\M_f$ in Theorem~\ref{thm:A} is identified with the KL basis of $\Hy_{S_d}$. See Example~\ref{ex:BA}(2) where an arbitrary parabolic KL basis of type A arises as a canonical basis of type B.

\subsection{$\imath$Schur duality}

Let $\V$ be the natural representation of the Drinfeld-Jimbo quantum group $\U=\U_q(\mathfrak {sl}_{2r+\nb})$. Let $(\U, \Ui)$ be the quantum symmetric pair of type AIII formulated by G.~Letzter \cite{Let99, Let02}, where $\Ui$ is a coideal subalgebra of $\U$ whose $q\mapsto 1$ limit is the enveloping algebra of $\mathfrak{sl}(r+m) \oplus \mathfrak{gl}(r)$; we shall refer to $\Ui$ as an $\imath$quantum group. When $\V$ is viewed as a representation of $\Ui$, its standard basis $\{v_i |  i\in \Ibw\}$ is naturally bicolored (where the $m$ indices in the middle are colored as $\bu$, while the remaining $2r$ indices are colored as $\circ$). When $\nb=0$ or $1$, $\Ui$ is quasi-split, and on the other extreme when $r=0$, we have $\Ui=\U$. 

We endow the tensor space $\V^{\otimes d}$ with a (right) $\HB$-module structure. The aforementioned $\HB$-modules $\M_f$ arise as direct summands of the tensor module $\V^{\otimes d}$ of $\HB$, and are called {\em quasi-permutation modules}. Each $\M_f$ is spanned by a standard basis $M_g$ where $g$ runs over a $W_d$-orbit. (We have chosen to parametrize $\M_f$ by ``anti-dominant weights" $f$.) 

Our second main result is the following.

\begin{customthm} {\bf B} [Theorem~\ref{thm:UiHB}]
  \label{thm:B}
The actions of $\Ui$ and $\HB$ on $\V^{\otimes d}$ commute with each other, and form double centralizers.
\end{customthm} 

The $\imath$quantum group $\Ui$ comes with parameters \cite{Let02}, and for our purpose, the 
parameters are fixed once for all by the double centralizer property in Theorem~\ref{thm:B}. 

Note that in the extreme case when $r=0$ and $\Ui=\U$, we (somewhat surprisingly) claim to have an action on $\V^{\otimes d}$ by $\HB$, not by $\Hy_{S_d}$ which one is familiar with. The puzzle is resolved when we note that the action of the generator $H_0$ of $\HB$ reduces to $p\cdot \text{Id}$, and we recover Jimbo duality \cite{Jim86} (q-Schur duality of type A) in disguise in this extreme case. On the other hand, when $m=0$ or $1$, $(\U, \Ui)$ is quasi-split, and we recover the (quasi-split) $\imath$Schur duality due to \cite{BW18a} for $p=q$ (and generalized to $p=1$ in \cite{Bao17} and to general $p$ in \cite{BWW18}). The action of $H_0$ in general is a suitable mixture of the actions in the 2 special cases. 

In the setting of Jimbo duality, the generators of Hecke algebra $\Hy_{S_d}$ were realized via the R-matrix \cite{Jim86}. In the quasi-split $\imath$Schur duality, the action of the additional generator $H_0$  of $\HB$ was realized via the K-matrix by Bao and the second author \cite[Theorems 2.18, 5.4]{BW18a}  (this is the first construction of a K-matrix built on the notion of an intertwiner or quasi K-matrix therein); see also \cite{BWW18}. We show that the action of $H_0$ in the setting of Theorem~\ref{thm:B} is again realized by a K-matrix, which has been available in greater generality in Balagovic-Kolb \cite{BK19}. This can be viewed as a distinguished example that the K-matrix provides solutions to the reflection equation, a property of the K-matrix in general as established in \cite{BK19}. 

\subsection{Compatible canonical bases}

Generalizing Lusztig's approach on canonical basis in \cite{Lus92, Lus93}, Bao and the second author \cite{BW18a, BW18b} have developed a theory of $\imath$canonical basis for $\imath$quantum groups arising from quantum symmetric pairs. We showed that any based module $M$ of a quantum group of finite type (cf. \cite[Chapter ~27]{Lus93}) when viewed as a module over an $\imath$quantum group with suitable parameters can be endowed with a new bar map $\iba$ and a distinguished $\iba$-invariant basis (called $\imath$canonical basis); this construction in particular applies to the quantum symmetric pair $(\U, \Ui)$ of type AIII, and $M =\V^{\otimes d}$, as in the setting of Theorem~\ref{thm:B}. 
Denote by $\{C_g \mid g\in \Ibw^d\}$ and $\{C^*_g \mid g\in \Ibw^d\}$ the $\imath$canonical and dual $\imath$canonical basis on $\V^{\otimes d}$. 


\begin{customthm} {\bf C} [Proposition~\ref{prop:ioba}, Theorem~\ref{thm:iCBsame}]
   \label{thm:C}
(1) There exists a bar involution on $\V^{\otimes d}$ which is compatible with the bar involutions on $\Ui$ and $\HB$. 

(2) The (dual) $\imath$canonical basis on $\V^{\otimes d}$ viewed as a $\Ui$-module coincide with the (dual) quasi-parabolic KL basis on $\V^{\otimes d}$ viewed as an $\HB$-module (see Theorem~\ref{thm:A}. 
\end{customthm}

In the extreme case when $r=0$ and $\Ui=\U$ (i.e., in the setting of \cite{Jim86}), Theorem~\ref{thm:C} recovers the main result of I.~Frenkel, Khovanov and Kirillov \cite{FKK98}. In the special case when $m=0$ or $1$, it reduces to the (quasi-split) $\imath$Schur duality in \cite{BW18a} (as well as the generalizations in \cite{Bao17, BWW18}). In the general case (for arbitrary $r$ and $m$), the $\imath$canonical basis elements in $\V^{\otimes d}$ parameterized by all black nodes $\bu$ (respectively, by all white nodes $\circ$) can be identified with parabolic KL of type A (respectively, B), but there are other $\imath$canonical basis elements of mixed colors without such identifications.

\subsection{An inversion formula}

An inversion formula for KL polynomials originated in \cite{KL79} and was subsequently generalized to the parabolic setting by Douglass \cite{Do90}; also see \cite{So97} for an exposition. 
In type A, the inversion formula can be reformulated and reproved using a symmetric bilinear form on the tensor product $\U$-module $\V^{\otimes d}$; see Brundan \cite{Br06} and Cao-Lam \cite{CL16}. We generalize the approach in \cite{CL16} via the $\imath$Schur duality by formulating a bilinear form $\la \cdot, \cdot \ra$ on $\V^{\otimes d}$ as a $\Ui$-module. 

\begin{customthm} {\bf D} 
[Theorems~\ref{thm:bisym}--\ref{thm:inv}]
  \label{thm:D}
(1) The bilinear form $\la\cdot,\cdot\ra$ on $\V^{\otimes d}$ is symmetric.

(2) The $\imath$canonical basis and dual $\imath$canonical basis on $\V^{\otimes d}$ are dual with respect to $\la\cdot,\cdot\ra$, i.e., 
$\la C_{g},C^*_{-h}\ra=\delta_{g,h}$, for $g,h\in f\cdot W_d.$
\end{customthm}
Theorem~\ref{thm:D} can be reformulated as a duality between (dual) quasi-parabolic KL polynomials; see Corollary~\ref{cor:inv}. It can be extended easily to a useful duality between KL polynomials of super type BCD introduced in \cite{BW18a, Bao17}; see Remark~\ref{rem:superKL}. The proof of Theorem~\ref{thm:D}(1) uses some old and new properties of the quasi R-matrix $\Theta^\io$ introduced in \cite{BW18a} (and generalized by Kolb \cite{Ko20}) and an anti-involution $\sigma_i$ on $\Ui$ in \cite{BW21}. 

\subsection{Related works}

Different constructions of Hecke algebra modules $\M_f$ appeared in earlier works of Dipper-James-Mathas \cite{DJM98} and Du-Scott \cite{DS00}, independently. To construct $q$-Schur algebras with desired homological (such as quasi-hereditary) property, these authors were led to consider generalized $q$-permutation (i.e., quasi-permutation) modules of Hecke algebra of type B associated to cosets $W_f\backslash W_d$, with $W_f$ as in \eqref{eq:Wf1}. In their approaches, such a module is defined to be a right ideal of $\HB$ generated by a single generator, say $x_\lambda$. The elements $x_\lambda$ constructed via Jucy-Murphy elements are not bar invariant in general (in contrast to the bar-invariance of the generator $M_f$ of $\M_f$ in our construction); see Remark~\ref{rem:DJMDS}. 

It is natural for us to formulate the $q$-Schur algebras $\sch(r| \nb|r,d) =\End_{\HB}(\V^{\otimes d})$, which now depend on 3 integers $r, \nb, d$; these are close cousins of the $(Q,q)$-Schur (or $q$-Schur$^2$) algebras  in \cite{DJM98, DS00}, which depend on 2 integers. These algebras include various $q$-Schur algebras in \cite{DJ89, Gr97, BKLW18, LL21} as special cases by setting $r=0$ or $\nb=0,1$. The basis theorem established in \cite{DJM98, DS00} on $\text{Hom}_{\HB} (\M_f, \M_{f'})$ provides us a basis for $\sch(r| \nb|r,d)$. There is also a generalization of $q$-Schur algebras in a different direction which is valid for Hecke algebras of all finite types in \cite{LW22}. 

\subsection{Further directions}

Let us give brief comments on several directions in which one can extend this work. We hope to return to some of these topics elsewhere. 

We plan to explore the Hecke modules and quasi-parabolic KL bases associated to reflection subgroups of Weyl groups and Coxeter groups in greater generalities. Reflection subgroups of Weyl groups are abundant, and affine Weyl groups offer exciting new possibilities. For now we are aware that similar constructions make sense in some cases beyond type B though the level of generalities remains to be clarified.  

Further $\imath$Schur dualities in connection to $\imath$canonical bases can be formulated in the setting of quantum symmetric pairs of classical (super) finite or affine types; this will be developed elsewhere. 

There exists a bar involution on the $q$-Schur algebra $\sch(r| \nb|r,d)$ induced from the bar involutions on quasi-permutation modules $\M_f$. It will be interesting to develop a theory of canonical basis on $\sch(r| \nb|r,d)$ and study its relation to the $\imath$canonical basis on the modified $\imath$quantum group (compare \cite{BLM90, BKLW18, LiW18, LW22}). 

The Kazhdan-Lusztig bases (as well as canonical bases arising from Jimbo-Schur duality and quasi-split $\imath$Schur duality) afford geometric interpretations in terms of flag varieties \cite{KL80, BLM90, GL92, BKLW18, LiW18}. 
It will be of great importance if one finds a geometric setting for the quasi-parabolic  KL bases (as well as for the $\imath$Schur duality), and this might well stimulate  a construction of new $\imath$quiver varieties.

\subsection{Organization} 

This paper is organized as follows. The action of the Hecke algebra $\HB$ on the tensor space $\V^{\otimes d}$ is formulated in Section~\ref{sec:tensor}. We develop in Section~\ref{sec:Mf} properties for the minimal length representatives of $W_f$ in the Weyl group $W_d$. We then construct the bar involution and canonical basis on the module $\M_f$, proving Theorem~\ref{thm:A}. 

In Section~\ref{sec:duality}, we recall the $\imath$quantum group $\Ui$ and set up the $\imath$Schur duality between $\Ui$ and $\HB$ acting on $\V^{\otimes d}$; see Theorem~\ref{thm:B}. In Section~\ref{sec:iCB}, the bar involutions on $\Ui$, $\V^{\otimes d}$, and $\HB$ are shown to be compatible. We then show that the $\imath$canonical basis on $\V^{\otimes d}$ as a $\Ui$-module coincides with the canonical basis on it as an $\HB$-module, proving Theorem~\ref{thm:C}. In Section~\ref{sec:inversion}, we establish Theorem~\ref{thm:D} on an inversion formula for quasi-parabolic KL polynomials.

\vspace{2mm}

{\bf Acknowledgement.} We thank George Lusztig for insightful comments and suggestions, and thank Li Luo for helpful remarks. YS is partially supported by a GSAS fellowship at University of Virginia, and WW is partially supported by the NSF grant DMS-2001351.

\section{Modules over Hecke algebra of type B}
 \label{sec:tensor}

In this section we introduce the Hecke algebra $\HB$ of type B and its action on a tensor space. This leads to quasi-permutation modules of $\HB$.

\subsection{Weyl group and Hecke algebra of type B}

The Weyl group $W=W_d$ of type $B_d$ is generated by $s_i$, for $0\le i \le d-1$, subject to the Coxeter relations: $s_i^2=1, (s_is_{i+1})^3=1, (s_0s_1)^4=1$, and $(s_i s_j)^2=1  \ (|i-j|>1)$. The symmetric group $S_d$ is a subgroup of $W_d$ generated by $s_i$, for $1\le i \le d-1$. Denote by $\N$ the set of non-negative integers. The length function $l: W_d \rightarrow \N$ is defined such that $l(\sigma) =k$ if $\sigma $ has a reduced expression $\sigma =s_{i_1} \cdots s_{i_k}$. 

For a real number $x\in \mathbb R$ and $m \in \N$, we denote $[x, x+m] =\{x, x+1, \ldots, x+m \}$. 
For $a \in \Z_{\ge 1}$, we denote by 
\[
\I_a = \left [\frac{1-a}2, \frac{a-1}2 \right].
\]
For $r, \nb \in \N$ (not both zero), we introduce a new notation for $\I_{2r+m}$ to indicate a fixed set partition:
\begin{equation}
  \label{eq:Ibw}
\Ibw : =\I_{2r+m}, \qquad
\Ibw =\Iwl \cup \Ib \cup \Iwr
\end{equation}
where the subsets
\begin{align}
 \label{eq:III}
\Iwr = \left [\frac{\nb+1}{2}, r+\frac{\nb-1}{2} \right].
\qquad
\Ib = \left [\frac{1-\nb}{2}, \frac{\nb-1}{2} \right], 
\qquad
\Iwl = - \Iwr,
\end{align}
have cardinalities $r, m, r$, respectively.

We view $f  \in \Ibw^d$ as a map $f: \{1, \ldots, d\} \rightarrow \Ibw$, and identify 
$f=(f(1), \ldots, f(d))$, with $f(i) \in \Ibw$. 
We define a right action of the Weyl group $W_d$ on $\Ibw^d$ such that,
for $f\in \Ibw^d$ and $0\leq j\leq d-1$,
\begin{equation}
  \label{eq:WB}
f^{s_j} =f\cdot s_j = 
\begin{cases} 
 (\cdots,f(j+1),f(j),\cdots),&\text{ if } j>0; \\
 (-f(1),f(2),\cdots,f(d)), &\text{ if } j=0,\ f(1)\in \Iwl\cup \Iwr; \\
 (f(1),f(2),\cdots,f(d)), &\text{ if } j=0,\ f(1)\in \Ib.
\end{cases} 
\end{equation} 
The only nontrivial relation $(s_0s_1)^4=1$ can be verified by case-by-case inspection depending on whether or not $f(1), f(2) \in \Ib$.
We sometimes write 
$$
f^\sigma = f\cdot \sigma 
=(f(\sigma(1)),\cdots,f(\sigma(d))),$$
where it is understood that 
\[
f(\sigma(i)) =
 \begin{cases}
 f(\sigma(i)), &\text{ if } \sigma(i)>0; \\
 f(-\sigma(i)), &\text{ if } \sigma(i)<0, f(-\sigma(i))\in \Ib; \\
 -f(-\sigma(i)), &\text{ if } \sigma(i)<0,f(-\sigma(i))\in \Iwl\cup \Iwr.
\end{cases}
\]

Let $p, q$ be two indeterminates. 
We denote $q_i=q$ for $1\leq i\leq d-1$ and $q_0=p$. The Iwahori-Hecke algebra of type B, denoted by $\HB$,  is a $\Q(p,q)$-algebra generated by $H_0,H_1,\cdots,H_{d-1}$, subject to the following relations:
\begin{align*}
&(H_i-q_i)(H_i+q_i^{-1})=0,\ \ \ \ &\text{for } i\geq 0; \\
&H_iH_{i+1}H_i=H_{i+1}H_iH_{i+1},\ \ \ \ &\text{for } i\ge 1;\\
&H_iH_j=H_jH_i,\ \ \ \ &\text{for }|i-j|>1;\\
&H_0H_1H_0H_1=H_1H_0H_1H_0. 
\end{align*}
The subalgebra generated by $H_i$, for $1\le i \le d-1$, can be identified with Hecke algebra $\Hy_{S_d}$ associated to the symmetric group $S_d$. If $\sigma \in W_d$ has a reduced expression $\sigma =s_{i_1} \cdots s_{i_k}$, we denote $H_\sigma =H_{i_1} \cdots H_{i_k}$. It is well known that $\{H_\sigma \mid \sigma \in W_d \}$ form a basis for $\HB$, and  $\{H_\sigma \mid \sigma \in S_d \}$ form a basis for $\Hy_{S_d}$. 

\subsection{A tensor module of $\HB$}

Consider the $\Q(p,q)$-vector space 
\begin{align}
\label{eq:V}
\V=\bigoplus_{a\in \Ibw}\Q(p,q)v_a. 
\end{align}
Given $f=(f(1), \ldots, f(d)) \in \Ibw^d$, we denote
\[
M_f = v_{f(1)} \otimes v_{f(2)} \otimes \ldots \otimes v_{f(d)}.
\]
We shall call $f$ a weight and $\{M_f \mid f\in \Ibw^d \}$ the standard basis for $\V^{\otimes d}$. 

In cases  $|\Ib| =0$ or $1$ (i.e., $m=0$ or $1$), the following lemma reduces to \cite[(6.8)]{BW18a} or \cite[(4.4)]{BWW18} in different notations. 

\begin{lemma}\label{lem:HB}
There is a right action of the Hecke algebra $\Hy_{B_d}$ on $\V^{\otimes d}$ as follows:
$$
M_f\cdot H_i=\left\{
\begin{aligned}
&M_{f\cdot s_i}+(q-q^{-1})M_{f},\ \ & \text{ if } f(i)<f(i+1),\ i>0;\\
&M_{f\cdot s_i},\ \ &\text{ if } f(i)>f(i+1),\ i>0; \\
&qM_f,\ \ & \text{ if } f(i)=f(i+1),\ i>0; \\
&M_{f\cdot s_i}+(p-p^{-1})M_f, \ \ &\text{ if } f(1)\in \Iwr,\ i=0; \\
&M_{f\cdot s_i},\ \ &\text{ if } f(1)\in \Iwl,\ i=0, \\
&pM_f,\ \ & \text{ if } f(1)\in \Ib,\ i=0.
\end{aligned}
\right.$$
\end{lemma}

\begin{proof}
It is a well known result of Jimbo \cite{Jim86} that the first 3 formulas above for $H_i$ with $i>0$ define a right action of Hecke algebra $\Hy_{S_d}$. 

It is clear that $(H_0-p)(H_0+p^{-1})=0$ and $H_0H_i =H_iH_0$, for $i\ge 2$. 

Hence, it remains to verify the braid relation $H_0H_1H_0H_1=H_1H_0H_1H_0$. To that end, we only need to consider the case $d=2$ and verify the braid relation when acting on $v_i\otimes v_j$.

If $i,j \in \Ib$, then $H_0$ acts on the span of $v_i\otimes v_j$ and $v_j \otimes v_i$ as $p\cdot \text{Id}$, and so the braid relation $H_0H_1H_0H_1=H_1H_0H_1H_0$ trivially holds. 

Assume now that at most one of $i,j$ lies in $\Ib$. If we formally regard this possible index in $\Ib$ as 0, then we are basically reduced to the setting of the action of Hecke algebra $\HB$ \cite[(6.8)]{BW18a} or \cite[(4.4)]{BWW18} (except a different partial ordering on $\Ibw^d$ was used therein, and $q,p$ here correspond to $q^{-1}, p^{-1}$ therein).
In any case, the braid relation can be verified directly case-by-case, and we provide some details below. 

For $i<j\in \Iwl$, we have
\begin{align*}
(v_i\otimes v_j) H_0H_1H_0H_1
&= v_{-i}\otimes v_{-j}+(q-q^{-1})v_{-j}\otimes v_{-i}
\\
&= (v_i\otimes v_j) H_1H_0H_1H_0. 
\end{align*}
For $i\in \Iwl,j\in \Ib$, we have
\begin{align*}
(v_i\otimes v_j) H_0H_1H_0H_1
&= p v_{-i}\otimes v_{j}+p(q-q^{-1})v_{j}\otimes v_{-i}
\\
&= (v_i\otimes v_j) H_1H_0H_1H_0.
\end{align*}
For $i\in \Iwl,j\in \Iwr$ such that $-i>j$, we have 
\begin{align*}
&(v_i\otimes v_j) H_0H_1H_0H_1
\\
&= v_{-i}\otimes v_{-j}+(q-q^{-1})v_{-j}\otimes v_{-i} +(p-p^{-1})v_{-i}\otimes v_{j}+(p-p^{-1})(q-q^{-1})v_{j}\otimes v_{-i}
\\
&= (v_i\otimes v_j) H_1H_0H_1H_0. 
\end{align*}
The remaining cases are similar and skipped.
\end{proof}

\subsection{Quasi-permutation modules}

Recall $\Ibw^d$ from \eqref{eq:Ibw}. A weight $f\in \Ibw^d$ is called {\em anti-dominant} if
\begin{align}  \label{def:adom}
\frac{m-1}{2}\geq f(1)\geq f(2)\geq \cdots \geq f(d).
\end{align}
Note that $f(j) \in \Iwl \cup \Ib$, for $1 \le j \le d$, if $f$ is anti-dominant. 
We denote 
\[
\Ianti=\{f\in \Ibw^d \mid f \text{ is anti-dominant} \}.
\] 

We can decompose $\V^{\otimes d}$ into a direct sum of cyclic submodules generated by $M_f$, for  anti-dominant weights $f$, as follows:
\begin{align}
  \label{eq:decomp}
\V^{\otimes d}=\bigoplus_{f \in \Ianti} \M_f, \qquad \text{ where }\;  \M_f =M_f\HB.
\end{align}
Denote by $\mathcal O_f$ the orbit of $f$ under the action of $W_d$ on $\Ibw^d$. 
The following is immediate from the formulas for the action of $\HB$ in Lemma~\ref{lem:HB}.  

\begin{lemma}
\label{lem:basisMf}
The right $\HB$-module $\M_f$ admits a $\Q(q)$-basis $\{M_g \mid g \in \mathcal O_f \}$. (It will be called the {\em standard basis}.)
\end{lemma}

By \eqref{def:adom}, we can suppose that $f \in \Ianti$ is of the form
\begin{align}
  \label{eq:f}
f= 
(\underbrace{a_1,\ldots,a_1}_{m_1},\ldots,\underbrace{a_k,\ldots,a_k}_{m_k},\underbrace{a_{k+1},\ldots,a_{k+1}}_{m_{k+1}},\ldots,\underbrace{a_l,\ldots,a_l}_{m_l}),
\end{align}
where $a_1> \ldots >a_k>a_{k+1}> \ldots > a_l,\ \{a_1,\ldots,a_k\}\subset \Ib,\ \{a_{k+1},\ldots, a_l\}\subset \Iwl$, and $m_1+\ldots+m_l=d$. The stabilizer subgroup of $f$ in $W_d$ is 
\begin{align}
  \label{eq:Wf}
W_f =W_{{m_1}}\times \ldots \times W_{{m_k}}\times S_{m_{k+1}}\times \ldots\times S_{m_l}. \end{align} 
Note the stabilizer subgroup $W_f$ is not a parabolic subgroup of $W_d$ when $2$ or more of the integers $m_1, \ldots, m_k$ are positive. (This phenomenon does not occur in the setting of \cite{BW18a, BWW18}.) 
We shall call the summand $\M_f$ in \eqref{eq:decomp} {\em quasi-permutation modules}. Clearly, for $f,f' \in \Ianti$, we have
\[
\M_f \cong \M_{f'}, \qquad \text{ if }\; W_f =W_{f'}.
\]  
If $W_f$ is not parabolic, $\M_f$ is in general not an induced module as those considered in parabolic Kazhdan-Lusztig theory \cite{De87}; see \cite{So97, LW23}. 

\begin{remark}
 \label{rem:DJMDS}
The quasi-permutation modules have appeared earlier in different formulations in \cite{DJM98} and~\cite{DS00}  independently. In our setting it is straightforward to write down the Hecke action and bases for the quasi-permutation modules $\M_f$ starting from $\V^{\otimes d}$, but it takes some nontrivial efforts to achieve this in \cite{DJM98, DS00}. In their approaches, the $q$-permutation modules are cyclic submodules of the right regular representation of $\HB$ with generators constructed by Jucys-Murphy elements. The quasi-permutation modules here are isomorphic to those {\em loc. cit.} integrally; this follows by comparing the formulas in Lemma~\ref{lem:HB} and \eqref{eq:decomp} with those in \cite[Lemmas~ 3.9, 3.11]{DJM98}. 
\end{remark}

\section{Canonical bases on quasi-permutation modules}
 \label{sec:Mf}

In this section, the minimal length representatives of the reflection subgroup $W_f$ of $W_d$ are studied. We construct a bar involution on the quasi-permutation modules $\M_f$ which are compatible with the bar involution on $\HB$. Then we construct a canonical basis on $\M_f$. 

\subsection{Basic properties of $W_d$}
\label{sub:permutation}

There is a natural left action of the Weyl group $W_d$ on the set 
$$
[\pm d] :=\{-d, \ldots,-2, -1,1,2,\ldots,d\}.
$$
such that 
\[
\sigma(-i)=-\sigma(i),\qquad \forall \sigma\in W_d, \, i \in [\pm d].
\]
In one line notation we write 
$$
\sigma=[\sigma(1),\ldots,\sigma(d)].
$$

Let $f\in \Ianti$. The stabilizer of $f$ in the symmetric group $S_d$ is always a parabolic subgroup generated by some subset $J(f) \subset \{s_1,\ldots, s_{d-1} \}$. We continue the notation \eqref{eq:f} for $f\in \Ianti$. Denote
\begin{align}  \label{eq:d}
d_\bu=m_1+\ldots+m_k, \qquad
d_\circ=d-d_\bu. 
\end{align}
That is, among $f(j)$, for $1\le j\le d$, the first $d_\bu$ of them belong to $\Ib$. 
Denote 
\begin{equation}
 \label{eq:ti}
t_1=s_0, \qquad
t_i=s_{i-1}t_{i-1}s_{i-1}, 
\qquad \text{ for } 1\leq i\leq d.
\end{equation}
Then $t_{i}$ is the swap (sign change) of $i$ and $-i$ while fixing $j \in [\pm d]$ with $j \neq \pm i$. 

\begin{lemma} \label{3.15}
Let $f\in \Ianti$. 
Then the stabilizer $W_f$ in $W_d$ is generated by 
\[
J_f :=\{t_i\mid 1\leq i\leq d_\bu\}\cup J(f).
\]
\end{lemma}

\begin{proof}
Recall $f$ from \eqref{eq:f}. The lemma follows since elements in $W_f$ are compositions of permutations in $S_d$ that fix $f$ and sign changes that fix each $a_j, 1\leq j \leq k$. 
\end{proof}

For $\sigma \in W_d$, the type B inversion number $\inv_B(\sigma)$ is defined to be (cf. \cite{BB05}) 
\begin{align}  \label{invB}
\inv_B(\sigma)=\inv(\sigma)+n_B(\sigma),
\end{align}
where 
\begin{align}
&\inv(\sigma) 
=\#\{(i,j)\mid 1\le i<j \le d,\sigma(i)>\sigma(j)\};
\label{invA} \\
&n_B(\sigma)=-\sum_{\{1\leq j\leq d\mid \sigma(j)<0\}}\sigma(j).
\label{nB}
\end{align}
For $\sigma\in S_d$, $\inv_B(\sigma)=\inv(\sigma)$ coincides with the inversion number of $S_d$. 

\begin{lemma}\cite[Proposition 8.1.1]{BB05}  \label{BB}
For any $\sigma\in W_d$, we have $l(\sigma)=\inv_B(\sigma)$.
\end{lemma}

\subsection{Minimal length representatives}

Let $f\in \Ianti$. Recall the stabilizer subgroup $W_f$ \eqref{eq:Wf} of $W_d$ is a (not-necessarily parabolic) reflection subgroup in general. 

\begin{lemma}  \cite[Lemma 1.9]{Lus84} \cite[Theorem~ 2.2.5]{DS00}
Every right coset of $W_f$ in the Weyl group $W_d$ has a unique minimal length representative. 
\end{lemma}
Denote by ${}^fW$ the set of minimal length right coset representatives for $W_f$ in $W_d$, for $f\in \Ianti$. We shall establish a basic property for ${}^fW$. 

\begin{lemma}\label{black}
Let $1\le i \le d$ and $\sigma  \in {}^fW$. If $|\sigma(i)|\leq d_\bu$, then $\sigma(i)>0$.
\end{lemma}

\begin{proof}
We prove by contradiction. Suppose this were not true, then there exists $1\le i_\bu \le d$ such that $\sigma(i_\bu)<0$ and $u_\bu=|\sigma(i_\bu)|\leq d_\bu$. By Lemma~\ref{3.15} we have $ t_{u_\bu} \in W_f$ and thus $ t_{u_\bu}\sigma \in W_f\sigma$. Now by \eqref{nB} we have $n_B(t_{u_\bu}\sigma )=n_B(\sigma)-u_\bu.$
On the other hand, since there are at most $u_\bu-1$ indices less than $u_\bu$, we have $\inv(t_{u_\bu}\sigma )\leq \inv(\sigma)+u_\bu-1$. 
Hence by the above 2 identities, \eqref{invB} and Lemma~\ref{BB}, we have 
\begin{align*}
l(t_{u_\bu}\sigma)
& 
=\inv(t_{u_\bu}\sigma ) +n_B(t_{u_\bu}\sigma )
\\
& \leq \inv(\sigma ) +n_B(\sigma ) -1
=l(\sigma)-1,
\end{align*}
which is a contradiction to the minimal length property of $\sigma$.
\end{proof}

\begin{example}
If $W_f$ is non-parabolic, the equality $l(ww') =l(w)+l(w')$ may fail for $w \in W_f$ and $w' \in {}^fW$. For example, take $W_f =\langle s_0, s_1s_0s_1 \rangle \subset W_{B_2}$ and $s_1$ is the minimal length representative of $W_f s_1$. Note $(s_1s_0s_1) s_{1} =s_1s_0$, but $l(s_1s_0s_1) + l(s_{1}) =4 \neq 2 =l(s_1s_0)$. 
\end{example}

The example above indicates \cite[Lemma 2.1(i)-(ii)]{De87} may fail for non-parabolic reflection subgroups. 
The next theorem, which is a generalization of \cite[Lemma 2.1(iii)]{De87} to reflection subgroups, is more difficult to establish. It will play a key role in constructing the bar involution and canonical bases for quasi-permutation modules. 

\begin{theorem}
   \label{thm:Wf}
Let $\sigma\in {}^fW$, and $0\le i \le d-1$. Then exactly one of the following possibilities occurs:
\begin{enumerate}
\item[(i)]
 $l(\sigma s_i)<l(\sigma)$. In this case, $\sigma s_i  \in {}^fW$; 
\item[(ii)]
 $l(\sigma s_i)>l(\sigma)$ and $\sigma s_i \in {}^fW$; 
\item[(iii)]
 $l(\sigma s_i)>l(\sigma)$ and $\sigma s_i \notin {}^fW$, for $i \neq 0$. In this case, $\sigma s_i=s' \sigma$, for some $s'\in J(f)$; 
\item[(iii$_0$)]
 $l(\sigma s_0)>l(\sigma)$ and $ \sigma s_0 \notin {}^fW$. In this case, $\sigma s_0=t \sigma$, for some $t \in J_f\backslash J(f)$.
\end{enumerate}
(More precisely, in case (iii), we have $f(\sigma(i)) =f(\sigma(i+1))$ and $s' =(|\sigma(i)|, |\sigma(i+1)|)$; in case (iii$_0$),  $\sigma(1)>0$ and $t =t_{\sigma(1)}$.)
\end{theorem}

\begin{proof}
We shall compare $\sigma \in {}^fW$ with $\sigma s_i$. Our argument below uses the action of $W_d$ on $\V^{\otimes d}$ crucially. We separate the proof into 2 cases depending on whether or not $i=0$. 
\vspace{2mm}

(1) Assume $i=0$. We separate into 3 subcases (i$_0$)-(iii$_0$) below by the range of $f^\sigma(1)$. 

\vspace{2mm}

(i$_0$) $\underline{f^{\sigma}(1)\in \Iwr  \Rightarrow \text{Case (i) for }i=0}$. 

In this case, we have $\sigma(1)<0$ since $f(\sigma(1)) = f^{\sigma}(1)\in \Iwr$ while $f(j) \not \in \Iwr$ (for $1\le j\le d$) thanks to $f$ being anti-dominant. 

{\bf Claim 1.} $l(\sigma s_0)=l(\sigma)-1$. 

Indeed, by Lemma~\ref{BB} it suffices to show that $\inv_B(\sigma s_0) < \inv_B(\sigma)$. 
Note that $\sigma s_0(j) =\sigma (j)$, for $2\le j \le d$, and $\sigma s_0(1) >0 > \sigma (1)$. By \eqref{nB} we have $n_B(\sigma s_0)=n_B(\sigma)+\sigma(1).$
On the other hand, we have $\inv(\sigma s_0)\leq \inv(\sigma)-\sigma(1)-1$ since there are at most $(-\sigma(1)-1)$ indices smaller than $-\sigma(1)$. Hence by \eqref{invB}, 
$\inv_B(\sigma s_0)\leq \inv_B(\sigma)-1,$ and Claim~1 follows.

It remains to verify that $\sigma s_0 \in {}^fW$. If this were not true, there exists $\tau\in  W_f\sigma s_0$ such that $l(\tau)<l(\sigma s_0)=l(\sigma)-1$. Hence $l(\tau s_0)\leq l(\tau)+1<l(\sigma)$; this is a contradiction since $\tau s_0\in W_f  \sigma $ and $\sigma$ is a minimal length representative of $W_f  \sigma $.

\vspace{2mm}

(ii$_0$) $\underline{f^\sigma(1)\in \Iwl  \Rightarrow \text{Case (ii) for }i=0}$. 

In this case, $f^{\sigma s_0}(1)\in \Iwr$, and $\sigma(1)>0$, thanks to $f$ being anti-dominant. Arguing as in (i$_0$) for Claim~1, we have $l(\sigma s_0 )=l(\sigma)+1$.
It remains to verify that $ \sigma s_0 \in {}^fW$. If this were not true, we choose the minimal length representative $\tau\in W_f\sigma s_0$. Since $\tau\in {}^fW$ and $f^\tau(1)\in \Iwr$, by (i$_0$) we know that $l(\tau s_0)=l(\tau)-1<l(\sigma s_0)-1=l(\sigma)$; this is a contradiction since $\tau s_0\in W_f \sigma$ and $\sigma$ is a minimal length representative of $W_f  \sigma $. 

\vspace{2mm}

(iii$_0$) $\underline{f^\sigma(1)\in \Ib \Rightarrow \text{Case (iii$_0$)}}$. 

Thanks to $f^\sigma(1)\in \Ib$, we obtain $f^{\sigma}=f^{\sigma s_0}$, that is, $\sigma s_0\in W_f\sigma $. Then $l(\sigma s_0)>l(\sigma)$ and $\sigma s_0 \notin {}^fW$, since $\sigma$ is a minimal length representative in $W_f\sigma$. Also, we have $\sigma s_0 \sigma ^{-1} =t_{|\sigma(1)|}$, and thus, $\sigma s_0=t_{|\sigma(1)|}\sigma$; cf. \eqref{eq:ti}. Since $f^{\sigma}(1)\in \Ib$, we have $|\sigma(1)|\leq d_\bu$; cf. \eqref{eq:d}. By Lemma~\ref{black}, we know that $\sigma(1)>0$. Hence, $t_{\sigma(1)} \in J_f\backslash J(f)$. 

\vspace{2mm}

(2) Assume $i>0$. We compare $\sigma \in {}^fW$ with $\sigma s_i$. By using inversion numbers, we see that $l(\sigma s_i)>l(\sigma)$ if and only if $f^{\sigma}(i)\geq f^{\sigma}(i+1)$. We separate into 3 subcases (i)-(iii) below depending on whether $f^{\sigma}(i)-f^{\sigma}(i+1)$ is negative, positive or zero.

\vspace{2mm}

(i) $\underline{(f^{\sigma}(i)<f^{\sigma}(i+1) ) \Rightarrow \text{Case (i)} \text{ for }i>0}$. 

In this case, $l(\sigma s_i)<l(\sigma)$. 
It remains to verify that $\sigma s_i \in {}^fW$. If this were not true, then there exists $\tau\in W_f\sigma s_i$ such that $l(\tau)<l(\sigma s_i)=l(\sigma)-1$. Thus $l(\tau s_i)\leq l(\tau )+1<l(\sigma)$; this is a contradiction since $\sigma$ has the minimal length and $\tau s_i\in W_f\sigma$. 

\vspace{2mm}

(ii) $\underline{(f^{\sigma}(i)>f^{\sigma}(i+1) ) \Rightarrow \text{Case (ii)} \text{ for }i>0}$. 

In this case, $l(\sigma s_i)>l(\sigma)$. 
Let us verify $\sigma s_i \in {}^fW$. If this were not true, choose the minimal length representative $\tau \in W_f\sigma s_i$. Since $f^{\tau }(i)<f^\tau (i+1)$, by (i) we have $l(\tau s_i)=l(\tau )-1<l(\sigma s_i)-1\leq l(\sigma)$, which is again a contradiction. 

\vspace{2mm}

(iii)  $\underline{(f^{\sigma}(i)=f^{\sigma}(i+1) )  \Rightarrow \text{Case (iii)}}$. 

In this case, $f^{\sigma s_i}=f^{\sigma}$, and $\sigma s_i\in W_f\sigma$. 
Without loss of generality we assume that $|\sigma(i)|<|\sigma(i+1)|$. 
It follows from the anti-dominance of $f$ that $\sigma(i)$ and $\sigma(i+1)$ have the same sign if $f^\sigma(i)=f^\sigma(i+1) \in \Iwl \cup \Iwr$;     
On the other hand,  if $f^\sigma(i) =f^\sigma(i+1) \in \Ib$, then $\sigma(i)$ and $\sigma(i+1)$ have the same $+$ sign by Lemma~\ref{black}. 

Therefore, we have $f(|\sigma(i)|)=f(|\sigma(i+1)|)$, and thus,
\begin{align}
\label{eq:conj}
\sigma s_i \sigma^{-1} = (|\sigma(i)|, |\sigma(i)|+1),
\end{align} 
that is,
\begin{align}
  \label{eq:ss}
\sigma s_i=s_{|\sigma(i)|}s_{|\sigma(i)|+1}\cdots s_{|\sigma(i+1)|-1}\cdots s_{|\sigma(i)|+1} s_{|\sigma(i)|}\sigma \in W_f \sigma. 
\end{align}
Since $f$ is anti-dominant (cf. \eqref{def:adom}), we must have 
\[
\{s_{|\sigma(i)|},s_{|\sigma(i)|+1},\cdots,s_{|\sigma(i+1)|-1}\}\subset J(f).
\]

{\bf Claim.} We have $|\sigma(i+1)|=|\sigma(i)|+1$.

Let us prove the Claim. Let $\sigma=s_1's_2'\cdots s_k'$ be a reduced expression. 
Assume to the contrary that $|\sigma(i+1)|>|\sigma(i)|+1$. Then we can reduce the length of the RHS of \eqref{eq:ss} by deleting a pair of simple reflections, at least one of which is some $s_i'$ from $\sigma$; otherwise, it would contradict the identity \eqref{eq:conj}. Now the element in the RHS of \eqref{eq:ss} after the deletion contradicts the minimality of $\sigma$ as a representative of $W_f \sigma$. Thus the Claim holds. 

Hence, setting $s' = (|\sigma(i)|, |\sigma(i+1)|) \in J(f)$, we have $s' \sigma=  \sigma s_i$.
\end{proof}

\begin{remark}
The conditions in Theorem~\ref{thm:Wf} have their counterparts in terms of $f^\sigma$ listed in the proof above, and they are useful in later applications. For instance, for $\sigma \in {}^fW$ and $i>0$, we have $\sigma s_i\in {}^fW$ if and only if $f^{\sigma}(i)\neq f^{\sigma}(i+1)$. 
\end{remark}

\subsection{The Hecke modules $\M_f$ revisited}

Recall the action of Hecke algebra on $\V^{\otimes d}$ from Lemma~\ref{lem:HB} and hence on $\M_f$ from \eqref{eq:decomp}. Applying Theorem~\ref{thm:Wf} and its proof, we shall obtain explicit descriptions for the action of the Hecke generators $H_i$ on the standard basis $\{M_{f\cdot \sigma}\mid \sigma \in {}^fW\}$ for $\M_f$, which is independent of the tensor module $\V^{\otimes d}$. Clearly, the length inequalities in Theorem~\ref{thm:Wf}  can be replaced by the Chevalley-Bruhat order $\le$ on $W_d$. 

\begin{proposition}
   \label{prop:HMf}
Let $f \in \Ianti$, $\sigma\in {}^fW$, and $0\le i \le d-1$. Then 
%
\begin{align*}
M_{f\cdot \sigma} H_i =
\begin{cases}
M_{f\cdot \sigma s_i} +(q_i -q_i^{-1}) M_{f\cdot \sigma}, & \text{ if }\ \sigma s_i<\sigma;
\\
M_{f\cdot \sigma s_i}, & \text{ if }\ \sigma s_i> \sigma \text{ and } \sigma s_i \in {}^fW;
\\
qM_{f\cdot \sigma}, & \text{ if } \  i\neq 0,\ \sigma s_i> \sigma \text{ and } \sigma s_i \not\in {}^fW;
\\
pM_{f\cdot \sigma}, & \text{ if }\   i= 0,\ \sigma s_0> \sigma \text{ and } \sigma s_0 \not\in {}^fW. 
\end{cases}
\end{align*}
\end{proposition}

\begin{proof}
In this proof we label the four cases in the proposition as (i), (ii), (iii), (iii$_0$), as they 
exactly correspond to the 4 cases in the same labelings in Theorem~\ref{thm:Wf}. 

We first assume $i\neq 0$. Then the cases (i), (ii), (iii) here match with the cases (i), (ii), (iii) in the proof of Theorem~\ref{thm:Wf} in the same order, which correspond to the 3 conditions $f^{\sigma}(i)<f^{\sigma}(i+1)$, $f^{\sigma}(i) >f^{\sigma}(i+1)$, and $f^{\sigma}(i) =f^{\sigma}(i+1)$ therein, respectively. Hence, the formulas in the proposition (with $i \neq 0$) follow by the first 3 formulas in Lemma~\ref{lem:HB}. 

Now we assume $i= 0$. Then the cases (i), (ii), (iii$_0$) here match with the cases (i$_0$), (ii$_0$), (iii$_0$) in the proof of Theorem~\ref{thm:Wf} in the same order, which correspond to the 3 conditions $f^{\sigma}(1) \in \Iwr$, $f^{\sigma}(1) \in \Iwl$, and $f^{\sigma}(1) \in \Ib$ therein, respectively. Hence the formulas in the proposition (with $i = 0$) follow by the last 3 formulas in Lemma~\ref{lem:HB}. 
%
%
%
\end{proof}

\begin{remark}
The formulas in Proposition~\ref{prop:HMf} miraculously take the same form as in the parabolic case \cite{De87, So97}. However, in contrast to {\em loc. cit.} it seems difficult to verify directly these formulas define a representation of $\HB$ in such a general reflection subgroup setting. The proof of Theorem~\ref{thm:Wf} provides us a crucial identification as {\em posets} between the orbit $f\cdot W_d$ (used in Lemma~\ref{lem:HB}) and the set of minimal length representatives ${}^fW$ for $W_f \backslash W_d$ (used in Proposition~\ref{prop:HMf}).
\end{remark}

\subsection{The bar involution on $\M_f$}

We prepare some lemmas toward the construction of the bar involution on $\M_f$. 

\begin{lemma}\label{2b}
For $f \in \Ianti$ and $\sigma \in {}^fW$, we have $M_fH_\sigma=M_{f\cdot \sigma}$.
\end{lemma}

\begin{proof}
We use induction on $l(\sigma)$. The case for $l(\sigma)=0$ is trivially true. 
 If $l(\sigma)=1$, then $\sigma =s_i$ for some $i$. If $i=0$, we have $f(1)\in \Iwl$, as otherwise we would have $s_0\in W_f$ (contradicting $\sigma =s_0 \in {}^fW$). Hence, $M_f H_0=M_{f\cdot s_0}$, by Lemma~\ref{lem:HB}. If $\sigma=s_i$ for $i>0$, we must have $f(i)>f(i+1)$. Thus $M_fH_i=M_{f\cdot s_i}$, again by Lemma~\ref{lem:HB}. 
 
Suppose $l(\sigma)>0$. We have a reduced expression $\sigma=s_{i_1}\cdots s_{i_k}$. Denote $\sigma'=s_{i_1}\cdots s_{i_{k-1}}$, and note  $l(\sigma') <l(\sigma)$. By Theorem~\ref{thm:Wf}(i), $\sigma' \in {}^fW$. By the inductive assumption, $M_fH_{\sigma'} =M_{f\cdot \sigma'}$.
Now if $s_{i_{k}}=s_0$, then  this only happens when $f^{\sigma'}(1)\in \Iwl$, by case~(i$_0$) in the proof of Theorem~\ref{thm:Wf}. Thus, we have $M_fH_{\sigma}=M_fH_{\sigma'} H_0=M_{f\cdot \sigma'}H_0=M_{f\cdot \sigma}$, by Lemma~\ref{lem:HB}. If $s_{i_{k}}=s_j$ for some $j\ge 1$, similarly we must have $f^{\sigma'}(j)>f^{\sigma'}(j+1)$, by case~(i) in the proof of Theorem~\ref{thm:Wf}. Thus we have $M_fH_{\sigma}=M_fH_{\sigma'} H_j=M_{f\cdot \sigma'}H_j=M_{f\cdot \sigma}$, again by Lemma~\ref{lem:HB}.
\end{proof}

\begin{lemma}\label{lem:re}
Suppose that $\sigma\in {}^fW$ satisfies that $1\neq | \sigma(1)| \leq d_\bu$. Then $\sigma(1)>1$, and $\sigma$ must have a reduced expression which starts with $s_{\sigma(1)-1}s_{\sigma(1)-2}\cdots s_2s_1$.
\end{lemma}

\begin{proof}
Lemma~\ref{black} is applicable by the assumption, and so we must have $\sigma(1)>0$, and then $\sigma(1)>1$, thanks to the assumption $1\neq | \sigma(1)|$.

Set $u=\sigma(1)$. We prove the lemma by induction on the length of $\sigma$. If $l(\sigma)=1$, then $\sigma=s_1$ (thanks to $\sigma(1)>1$), and the lemma holds trivially. 

Now suppose that  $l(\sigma)>1$. There exists $1\leq a\leq d$ such that $\sigma(a)=u-1$ by Lemma~\ref{black}. Then we have $s_{u-1}\sigma (1)=u-1, s_{u-1}\sigma (a)=u$ and thus \[
l(s_{u-1}\sigma )=\inv_B(s_{u-1}\sigma)=\inv_B(\sigma)-1=l(\sigma)-1.
\]
By the inductive assumption, $s_{u-1}\sigma$ has a reduced expression which starts with $s_{\sigma(1)-2}\cdots s_2s_1$.  Therefore, $\sigma$ has a reduced expression which starts with $s_{\sigma(1)-1}s_{\sigma(1)-2}\cdots s_2s_1$.
\end{proof}

The bar involution on $\mathscr H_{B_d}$, denoted by$~^-$, is the $\Q$-algebra automorphism such that
\[
\bar{H}_{i}=H_{i}^{-1},
\ \bar{q}=q^{-1},
\ \bar{p}=p^{-1},
\ \forall 0\le i \le d-1.
\]
(We shall refer to a map such that $q^m \mapsto q^{-m}$ and $p^m \mapsto p^{-m}$ {\em anti-linear}.)

Let $f \in \Ianti$. We define a $\Q$-linear map $\iba$ on the module $\M_f$ (which has a basis $M_{f\cdot \sigma}$, for $\sigma \in {}^fW$) by
\begin{align}
  \label{eq:barM}
\iba(q)=q^{-1},\quad \iba(p)=p^{-1},\quad \iba(M_{f \cdot \sigma})=M_f\bar H_{\sigma},\quad \forall \sigma \in {}^fW.
\end{align}

Now we can establish the existence of bar involution on $\M_f$, generalizing the parabolic case \cite{De87, So97}.

\begin{proposition}
  \label{prop:iHba}
Let $f \in \Ianti,$. The map $\iba$ on $\M_f$ in \eqref{eq:barM} is compatible with the bar operator on the Hecke algebra, i.e., 
\begin{align}
  \label{MxH}
\iba(xh)=\iba(x) \overline{h}, 
\qquad
\text{for all $x \in \M_f, \  h \in \HB$}.
\end{align}
In particular, $\iba^2 =\text{Id}$. (We shall call $\iba$ the bar involution on $\M_f$.)
\end{proposition}

\begin{proof}
Note $\iba(M_f) =M_f$, by definition \eqref{eq:barM}.

A simple induction on $l(w)$ reduces the proof of \eqref{MxH}, for $h =H_w$ with $w\in W_d$, to proving the following formula:
\begin{align}
  \label{MHH}
\iba(x H_i)=\iba(x)\bar H_i,
\qquad \text{ for all }  x\in \M_f, 0 \le i\le d-1.
\end{align}

It suffices to verify \eqref{MHH} for the basis elements of $\M_f$, $x = M_fH_{\sigma}$ (that is, $x=M_{f\cdot \sigma}$ by Lemma~\ref{2b}), for $\sigma \in {}^fW$. We proceed case-by-case following Theorem~\ref{thm:Wf}. 

(i) Assume $l(\sigma s_i)<l(\sigma)$. In this case $ \sigma s_i \in {}^fW$, and thus 
\begin{align*}
\iba(M_fH_{\sigma} H_i)
& =\iba(M_f H_{\sigma s_i} +(q_i-q_i^{-1}) M_f H_{\sigma}) \\
&= M_f \bar H_{\sigma s_i}+(q_i^{-1} -q_i) M_f \bar H_{\sigma} \\
&= M_f(\overline{H_{\sigma s_i}+(q_i -q_i^{-1})H_{\sigma}})
=M_f\bar H_{\sigma} \bar H_i
=\iba(M_fH_{\sigma})\bar H_i.
\end{align*}

(ii) If $l(\sigma s_i)>l(\sigma)$ and $\sigma s_i \in {}^fW$, then
\begin{align*}
\iba(M_fH_{\sigma}H_i)=\iba(M_fH_{\sigma s_i})
=\iba(M_f)\bar H_{\sigma s_i}
=\iba(M_f)\bar H_{\sigma} \bar{H_i}
=\iba(M_fH_{\sigma})\bar H_i. 
\end{align*}

(iii) Assume $l(\sigma s_i)>l(\sigma)$ and $ \sigma s_i\notin {}^fW$, for $i> 0$. In this case, we have $ \sigma s_i=s'\sigma $ for some $s'\in J(f)$, and $M_f H_{s'}=q M_f$ by Lemma~\ref{lem:HB}. Thus, we have 
\begin{align*}
\iba(M_fH_{\sigma} H_i)=\iba(M_fH_{\sigma s_i})=\iba(M_f H_{s'\sigma})=\iba(qM_fH_{\sigma})=q^{-1}M_f\bar H_{\sigma}.
\end{align*}   
On the other hand, we have
\begin{align*}
\iba(M_fH_{\sigma})\bar H_i
=M_f \bar{H}_{\sigma} \bar{H_i}
=M_f\bar H_{\sigma s_i}
=M_f\bar H_{s'\sigma}
=M_fH_{s'}^{-1}\bar H_{\sigma}
=q^{-1}M_f\bar H_{\sigma}.
\end{align*}
Hence \eqref{MHH} holds for $x =M_{f} H_{\sigma}$ in this case.

(iii$_0$) Assume $i=0$, $l(\sigma s_0)>l(\sigma)$, and $\sigma s_0\notin {}^fW$. By Theorem~\ref{thm:Wf}(iii$_0$) and its proof in case (iii$_0$), we have $f^{\sigma}(1)\in \Ib$ and thus $|\sigma(1)|\leq d_\bu$. By Lemma~\ref{black}, $\sigma(1)>0$. We separate into 2 subcases (iii$_0$-1) and (iii$_0$-2).

\underline{Subcase (iii$_0$-1):  ${\sigma(1)=1}$}. 
Then $f(1)\in \Ib$ and $s_0\sigma=\sigma s_0$, by Theorem~\ref{thm:Wf}(iii$_0$) and its proof  in case (iii$_0$). Thus we have 
\begin{align*}
\iba(M_fH_{\sigma} H_{0})=&\iba(M_fH_{\sigma s_0})=\iba(M_fH_{s_0\sigma})=\iba(M_fH_0H_{\sigma})=p^{-1}M_f\bar H_{\sigma}.
\end{align*}
On the other hand, 
$\iba(M_fH_{\sigma})\bar H_0=M_f\bar H_{\sigma s_0}=M_f \bar H_{s_0\sigma}=p^{-1}M_f\bar H_{\sigma}.$  
So $\iba(M_fH_{\sigma} H_{0})= \iba(M_fH_{\sigma})\bar H_0$, proving \eqref{MHH} for $x =M_{f} H_{\sigma}$ in this case.

\underline{Subcase (iii$_0$-2): ${\sigma(1) >1}$}. 
Set $u=\sigma(1) \le d_\bu$.
We have $\sigma s_0 =t_{u}\sigma$ by Theorem~\ref{thm:Wf}(iii$_0$); see \eqref{eq:ti} for $t_u$. By Lemma~\ref{lem:re}, $\sigma$ has a reduced expression of the form 
\[
\sigma=s_{u-1}s_{u-2}\cdots s_2s_1s_{i_1}\cdots s_{i_m}. 
\]
Hence, $t_u\sigma=s_{u-1}\cdots s_1s_0s_{i_1}\cdots s_{i_m}$, also a reduced expression for length reason. Thus  
\begin{align*}
\iba(M_fH_{\sigma}H_{0})
&= \iba(M_f H_{\sigma s_0}) =\iba(M_fH_{t_u\sigma}) \\
&= \iba(M_f H_{s_{u-1}\cdots s_1}H_{0} H_{s_{i_1}\cdots s_{i_m}}) \\
(u\leq d_\bu, \text{Lemma}~\ref{lem:HB} \text{ for }H_0)\Rightarrow \quad
&= p^{-1}\iba(M_f H_{s_{u-1}\cdots s_1} H_{s_{i_1}\cdots s_{i_m}}) \\
&= p^{-1}M_f\bar H_{\sigma}.
\end{align*}
On the other hand, we have
\begin{align*}
\iba(M_fH_{\sigma})\bar H_0
&= M_f\bar H_{\sigma} \bar{H_0}=M_f\bar H_{\sigma s_0}=M_f \bar H_{t_u\sigma} \\
&= M_f \overline{H_{s_{u-1}\cdots s_1}}\bar H_0 \overline{H_{s_{i_1}\cdots s_{i_m}} }  \\
(u\leq d_\bu, \text{Lemma}~\ref{lem:HB} \text{ for } H_0)\Rightarrow\quad 
&= p^{-1} M_f \overline{H_{s_{u-1}\cdots s_1}} \, \overline{H_{s_{i_1}\cdots s_{i_m}} }  \\
&= p^{-1}M_f\bar H_{\sigma}.
\end{align*}
Therefore, the proof of \eqref{MHH} is completed, for all $x =M_{f} H_{\sigma}$. 

Finally, we have 
$\iba^2 (M_fH_{\sigma})=M_f\bar{\bar{H}}_{\sigma}=M_fH_{\sigma}$, i.e.,  $\iba^2 = \text{Id}$.
\end{proof}

\begin{remark}
Recalling \eqref{eq:f} and \eqref{eq:d}, we define reflection subgroups
$
S_f := S_{{m_1}}\times \ldots \times 
 S_{m_l}$,
$S_f^\bu := S_{d_\bu} \times S_{m_{k+1}}\times \ldots\times S_{m_l}$,
and $W_f^\bu := W_{d_\bu} \times S_{m_{k+1}}\times \ldots\times S_{m_l}$ 
of $W_d$. Note that $S_f \stackrel{\text{par.}}{\subset} S_f^\bu \subset W_f^\bu \stackrel{\text{par.}}{\subset}  W_d$, where $\text{par.}$ stands for parabolic. 
Let us outline a $3$-step induction process of realizing (an isomorphic copy of) the $\HB$-module $\M_f$: first induce the 1-dimensional ``trivial" module from $\Hy_{S_f}$ to $\Hy_{S_f^\bu}$, then view the $\Hy_{S_f^\bu}$-module as an $\Hy_{W_f^\bu}$-module by imposing the action of $H_0$ as $p \cdot \text{Id}$, and finally induce once more from $\Hy_{W_f^\bu}$ to $\HB$. The bar involution on $\M_f$ can also be understood this way. We will not use this remark in this paper. This 3-step process can  be formalized and its generalization to other types will be treated in detail elsewhere. 
\end{remark}

\subsection{Canonical basis on $\M_f$}

For the formulation of canonical basis on $\M_f$, we shall specialize to a one-parameter setting. Our assumption below that $p\in q^\Z$ below amounts to choosing distinguished weight functions \`a  la Lusztig \cite{Lus03}. (The general weight functions therein work here too, but it would require additional notations to set up properly.)

Suppose $p \in q^\Z$. Then $\HB$ becomes a $\Q(q)$-algebra, and $\M_f$ becomes a $\Q(q)$-vector space and an $\HB$-module. The bar involution $\iba$ on $\M_f$ remain valid. 
With Proposition~\ref{prop:HMf} and Proposition~\ref{prop:iHba} at our disposal, the proof of the next theorem follows by standard arguments.

\begin{theorem}
  \label{thm:CBMf}
Suppose $p \in q^\Z$, and let $f\in \Ianti$. Then for each $\sigma \in {}^fW$, there exists a unique element $C_{ \sigma} \in \M_f$ such that
\begin{enumerate}
\item[(i)]
$\iba (C_{\sigma}) =C_{\sigma}$;
\item[(ii)]
$C_{\sigma} \in M_{f\cdot \sigma}+\sum_{w\in {}^fW }\limits 
q^{-1}\Z[q^{-1}] M_{f\cdot w}.$
\end{enumerate}
\noindent Moreover,  we have 
\begin{enumerate}
\item[(ii$'$)]  $C_{\sigma} \in M_{f\cdot \sigma} +\sum_{w\in {}^fW, w< \sigma}\limits 
q^{-1}\Z[q^{-1}] M_{f\cdot w}.$
\end{enumerate}
\end{theorem}
The set $\{ C_{\sigma} | \sigma \in {}^fW \}$ is called a {\em canonical basis or quasi-parabolic KL basis} for $\M_f$. 

\begin{proof}
Let $\sigma \in {}^fW$. Assume $p\in q^{\Z_{> 0}}$, and set $b_i =H_i +q_i^{-1}$, which is bar invariant. Proposition~\ref{prop:HMf} can be rewritten as
\begin{align}  \label{eq:MM}
M_{f\cdot \sigma} b_i =
\begin{cases}
M_{f\cdot \sigma s_i} + q_i M_{f\cdot \sigma}, & \text{ if }\ \sigma s_i<\sigma;
\\
M_{f\cdot \sigma s_i} + q_i^{-1} M_{f\cdot \sigma}, & \text{ if }\ \sigma s_i> \sigma \text{ and } \sigma s_i \in {}^fW;
\\
(q+q^{-1}) M_{f\cdot \sigma}, & \text{ if }\ \sigma s_i> \sigma \text{ and } \sigma s_i \not\in {}^fW,   i\neq 0;
\\
(p+p^{-1}) M_{f\cdot \sigma}, & \text{ if }\ \sigma s_0> \sigma \text{ and } \sigma s_0 \not\in {}^fW. 
\end{cases}
\end{align}
Now the existence of $C_{\sigma}$ satisfying Conditions~(i) and (ii$'$) can be proved using \eqref{eq:MM} by an induction on the Chevalley-Bruhat order for $\sigma$, following exactly the same argument as for \cite[Theorem 3.1]{So97}. 

(For $p\in q^{\Z_{< 0}}$, one reruns the argument therein by using a variant of \eqref{eq:MM} with $b_0 =H_0 -p$; for $p=1$, one uses $b_0=H_0$ instead.)

The uniqueness of the basis $\{C_{\sigma} \}$ follows from the following (cf. \cite{So97}).

{\bf Claim.} Suppose $z=\sum_{w\in {}^fW} h_wM_{f\cdot w}$ with all $h_w \in q^{-1}\Z[q^{-1}]$ satisfies $\iba(z)=z$. Then $z=0$. 

Indeed, if $z\neq 0$, we can choose $w'$ with maximal length such that $h_{w'} \neq 0$. Then it follows by the existence of $\{ C_{\sigma}\}$ satisfying (i) and (ii$'$) above and $z=\iba(z)$ that $h_{w'} =\bar h_{w'}$, which forces $h_{w'}=0$ (since $h_{w'} \in q^{-1}\Z[q^{-1}]$), which is a contradiction. The Claim follows. 
\end{proof}

Set $b_i' =H_i -q_i$. Proposition~\ref{prop:HMf}, for $f \in \Ianti, \sigma \in {}^fW$, can be rewritten as
\begin{align}  \label{eq:MM2}
M_{f\cdot \sigma} b_i' =
\begin{cases}
M_{f\cdot \sigma s_i} - q_i^{-1} M_{f\cdot \sigma}, & \text{ if }\ \sigma s_i<\sigma;
\\
M_{f\cdot \sigma s_i} - q_i M_{f\cdot \sigma}, & \text{ if }\ \sigma s_i> \sigma \text{ and } \sigma s_i \in {}^fW;
\\
0, & \text{ if }\ \sigma s_i> \sigma \text{ and } \sigma s_i \not\in {}^fW,   i\neq 0;
\\
0, & \text{ if }\ \sigma s_0> \sigma \text{ and } \sigma s_0 \not\in {}^fW. 
\end{cases}
\end{align}
The following counterpart of Theorem~\ref{thm:CBMf} (with $q^{-1}$ replaced by $q$) can be proved in the same way using \eqref{eq:MM2}. 

\begin{proposition}
  \label{prop:dualCBMf}
Suppose $p \in q^\Z$. There exists a basis $\{ C^*_{\sigma} | \sigma \in {}^fW \}$  (called dual canonical basis) for $\M_f$ which is characterized by 
$\iba (C^*_{\sigma}) =C^*_{\sigma}$ and 
$C^*_{\sigma} \in M_{f\cdot \sigma}+\sum_{w\in {}^fW}  q \Z[q] M_{f\cdot w}.$ 
Moreover, we have $C^*_{\sigma} \in M_{f\cdot \sigma}+\sum_{w\in {}^fW \atop w<\sigma} 
q \Z[q] M_{f\cdot w}.$
\end{proposition}
 The set $\{ C^*_{\sigma} | \sigma \in {}^fW \}$ is called a {\em dual canonical or dual quasi-parabolic KL basis} for $\M_f$. 

\begin{example}
\label{ex:BA}
  {\quad}
\begin{enumerate}
\item
If $f \in \Ianti$ satisfies $f(i) \in \Iwl$, for all $1\le i \le d$ (or more generally, if $k\le 1$ in \eqref{eq:f}--\eqref{eq:Wf}), then the subgroup ${}^fW$ is parabolic. In this case, the canonical basis of $\M_f$ is exactly the parabolic Kazhdan-Lusztig basis of type B \cite{KL79, De87}.

\item
If $f \in \Ianti$ satisfies $f(i) \in \Ib$, for all $1\le i \le d$, then the action of $H_0$ is given by $p \cdot \text{Id}$ on $\M_f$, and the $\HB$-module $\M_f$ essentially reduces to an $\Hy_{S_d}$-module. In this case, $W_f= B_{m_1}\times \ldots \times B_{m_k}$ with $m_1 +\ldots + m_k=d$, the canonical basis of $\M_f$ is identified with the parabolic KL basis of $\Hy_{S_d}$ associated to $(S_{m_1}\times \ldots \times S_{m_k})\backslash S_d$. (This  follows by the uniqueness of a canonical basis, since $\M_f$ as an $\HB$-module and as an $\Hy_{S_d}$-module has the same standard basis and the same bar map.)
\end{enumerate}
\end{example}

\begin{example}
 \label{ex:C}
For non-parabolic $W_f$, the canonical basis on $\M_f$ may not be a (usual) KL basis. 
Consider $\V^{\otimes 3}$ for $\V$ of dimension $5$ with standard basis $\{v_i\}_{-2\le i \le 2}$, where $\I_\bu=\{-1,0, 1\}$ (i.e., $\nb=3, r=1$ and $d=3$). 
We consider $f=(0,-1, -2)$ and $W_f =B_1 \times B_1  =\langle s_0,  s_{101}  \rangle$; here and below we shall write $s_is_js_k \cdots =s_{ijk\cdots}$. Then 
\[
{}^fW =\{e, s_1, s_2, s_{12}, s_{21}, s_{121}, s_{210}, s_{2101}, s_{1210}, 
 s_{12101}, s_{21012}, s_{121012}\}. 
\]
We have the following 12 canonical basis elements in $\M_f$ (as linear combinations of the 12 standard basis elements $M_{f\cdot \sigma}$, for $\sigma \in {}^fW$):
\begin{equation*}
\begin{aligned}
C_{f} &=M_f,
\quad  
C_{f\cdot s_1}=M_{f\cdot s_1} +q^{-1} M_f,
\quad 
C_{f\cdot s_2} =M_{f\cdot s_2} +q^{-1} M_f,\\
C_{f\cdot s_{12}}&=M_{f\cdot s_{12}}+q^{-1}M_{f\cdot s_{1}}+q^{-1}M_{f\cdot s_{2}}+q^{-2}M_{f},
 \\
C_{f\cdot s_{21}}&=M_{f\cdot s_{21}}+q^{-1}M_{f\cdot s_{2}}+q^{-1}M_{f\cdot s_{1}}+q^{-2}M_f, \\
C_{f\cdot s_{121}}&=M_{f\cdot s_{121}}+q^{-1}M_{f\cdot s_{12}}+q^{-1}M_{f\cdot s_{21}}+q^{-2}M_{f\cdot s_{1}}+q^{-2}M_{f\cdot s_{2}}+q^{-3}M_f,
 \\
C_{f\cdot s_{210}} &= M_{f\cdot s_{210}}+q^{-1}M_{f\cdot s_{21}}+q^{-2}M_{f\cdot s_{2}}+q^{-2}M_{f\cdot s_{1}}+(q^{-3}-q^{-1})M_f,
\\
C_{f\cdot s_{2101}} &= M_{f\cdot s_{2101}}+q^{-1}M_{f\cdot s_{210}}+q^{-2}M_{f\cdot s_{21}}\\
&\quad +(q^{-3}-q^{-1})M_{f\cdot s_{1}}+q^{-3}M_{f\cdot s_{2}}+(q^{-4}-q^{-2})M_f,
\\
C_{f\cdot s_{1210}} &= M_{f\cdot s_{1210}}+q^{-1}M_{f\cdot s_{210}}+q^{-1}M_{f\cdot s_{121}}+q^{-2}M_{f\cdot s_{21}}+q^{-2}M_{f\cdot s_{12}} \\
&\quad +q^{-3}M_{f\cdot s_{1}}+q^{-3}M_{f\cdot s_{2}}+q^{-4}M_f,
\\
C_{f\cdot s_{21012}}&= M_{f\cdot s_{21012}}+q^{-1}M_{f\cdot s_{2101}}+q^{-1}M_{f\cdot s_{1210}}+q^{-2}M_{f\cdot s_{210}}+q^{-2}M_{f\cdot s_{121}} \\
&\quad +q^{-3}M_{f\cdot s_{21}}+(q^{-3}-q^{-1})M_{f\cdot s_{12}}+(q^{-4}-q^{-2})M_{f\cdot s_{1}}+q^{-4}M_{f\cdot s_{2}} +q^{-5}M_f,
 \\
C_{f\cdot s_{12101}}&=M_{f\cdot s_{12101}}+q^{-1}M_{f\cdot s_{1210}}+q^{-1}M_{f\cdot s_{2101}}+q^{-2}M_{f\cdot s_{210}}+q^{-2}M_{f\cdot s_{121}}\\
&\quad +q^{-3}M_{f\cdot s_{21}}+q^{-3}M_{f\cdot s_{12}}+q^{-4}M_{f\cdot s_{2}}+q^{-4}M_{f\cdot s_{1}} +q^{-5}M_f,
 \\
C_{f\cdot s_{121012}}&= M_{f\cdot s_{121012}}+q^{-1}M_{f\cdot s_{21012}}+q^{-1}M_{f\cdot s_{12101}}+q^{-2}M_{f\cdot s_{2101}}+q^{-2}M_{f\cdot s_{1210}} \\
&\quad +q^{-3}M_{f\cdot s_{210}}+q^{-3}M_{f\cdot s_{121}}+q^{-4}M_{f\cdot s_{21}}+q^{-4}M_{f\cdot s_{12}}\\
&\quad +q^{-5}M_{f\cdot s_{2}}+q^{-5}M_{f\cdot s_{1}}+q^{-6}M_f.
\end{aligned}
\end{equation*}
Note that some polynomials in $q^{-1}$ above do not have positive coefficients in contrast to parabolic KL polynomials. Therefore, we do not expect a straightforward generalization of the geometric realization of the KL basis given in \cite{KL80}. 
\end{example}

\section{$\io$Schur duality of type AIII}
 \label{sec:duality}
 
In this section, we formulate a double centralizer property for the actions of $\Ui$ and $\HB$ on the tensor space $\V^{\otimes d}$. 

\subsection{Quantum group of type A} 

Denote the quantum integers and quantum binomial coefficients by, for $a\in \Z, k \in \N$,
\[
[a]=\frac{q^a-q^{-a}}{q-q^{-1}},
\qquad
\qbinom{a}{k} =\frac{[a] [a-1] \ldots [a-k+1]}{[k]!}.
\] 

For $r, m \in \N$ (as in the previous sections),  it is convenient to introduce 
\[
\pt=\frac{\nb}{2} \in \frac12 \N, 
\]
and denote
\[
I :=\I_{2r+2\pt-1} = \left [1-\pt -r, \pt+r-1 \right].
\]
Denote by $(a_{ij})_{i,j\in I}$ the Cartan matrix of type $A_{2r+\nb-1}$. 
For $i \neq j\in I$, let $S_{ij}(x,y)$ denote the noncommutative polynomial in two variables
\[
S_{ij}(x,y)=\sum_{s=0}^{1-a_{ij}}(-1)^s\qbinom{1-a_{ij}}{s} x^{1-a_{ij}-s}yx^s.
\]

The quantum group $\U=\U_q(\mathfrak {sl}_{2r+\nb})$ is a $\Q(q)$-algebra with generators $E_i,F_i, K_i^{\pm 1}$ ($i\in I$), subject to the standard defining relations including q-Serre relations 
\[
S_{ij}(E_i,E_j) =S_{ij}(F_i,F_j)=0,\quad  \text{ for }i\neq j \in I,
\]
 cf. \cite{Lus93, Jan95}. We define $K_\mu =\prod_{i} K_i^{a_i}$ for $\mu =\sum_i a_i i \in Y:=\Z I$. As an extension of a bar involution on $\Q(q)$ such that $\overline{q} =q^{-1}$, the bar involution $\psi$ on the algebra $\U$ is given by 
$\psi(q)=q^{-1},\ \ \psi(E_i)=E_i,\ \ \psi(F_i)=F_i,\ \ \psi(K_\mu)=K_{-\mu}.$

A comultiplication $\Delta$ on $\U$ is given by, for $i \in I,  \mu \in Y$, 
\begin{align}
\label{eq:Delta}
\Delta(E_i) =E_i\otimes 1+ K_i \otimes E_i, \,\,
\Delta(F_i) &=F_i\otimes  K_i^{-1}+1\otimes F_i,\,\,
\Delta(K_\mu) =K_\mu \otimes K_\mu.
\end{align}
The comultipication here follows \cite{Lus93}; it is consistent with \cite{BW18b} but different from the one used \cite{BW18a}. 

Denote the set of simple roots and the weight lattice for $\mathfrak {sl}_{2r+\nb}$ by 
\begin{align*}
\Pi &=\{\alpha_i=\epsilon_{i-\frac{1}{2}}-\epsilon_{i+\frac{1}{2}} \mid i\in I\},
\qquad
X =\bigoplus_{i\in \Ibw} \Z\epsilon_{i}.
\end{align*} 
Define the symmetric bilinear form on $X$, $(\cdot, \cdot): X \times X \rightarrow \Z$, such that $(\epsilon_i, \epsilon_j) =\delta_{ij}$. 

We also recall the braid group action $T_i=T_{i,+1}'': \U \rightarrow \U$ and its inverse from~\cite[5.2.1]{Lus93}, whose the action on $\U^+$ is given as follows: for  $i\neq j \in I$,  
\begin{equation}\label{eq:braid}
\begin{split}
\T_{i} (E_i) = -F_i K_{i}, \qquad
\T_{i} (E_j) &=  \sum_{r+s = - a_{ij}} (-1)^r q^{-r}_{i} E^{(s)}_i E_j E^{(r)}_i;\\ 
\T_{i}^{-1} (E_i) = - K_{i}^{-1} F_i, \qquad
\T_{i}^{-1} (E_j) &=  \sum_{r+s = - a_{ij}} (-1)^r q^{-r}_{i} E^{(r)}_i E_j E^{(s)}_i. 
\end{split}
\end{equation}
For any Weyl group element $w$, an automorphism $\T_w$ of $\U$ is defined via a reduced expression of $w$. This applies in particular to $w_0$, the longest element in the Weyl group of $\mathfrak {sl}_{2r+\nb}$. 

\subsection{$\io$Quantum group of type AIII} 
\label{notation}

Fix
\[
n =\frac{m}2 \in \frac12 \N. 
\]
We consider the Satake diagram of type AIII with $m-1 =2n-1$ black nodes and $r$ pairs of white nodes, together with a diagram involution $\tau$:
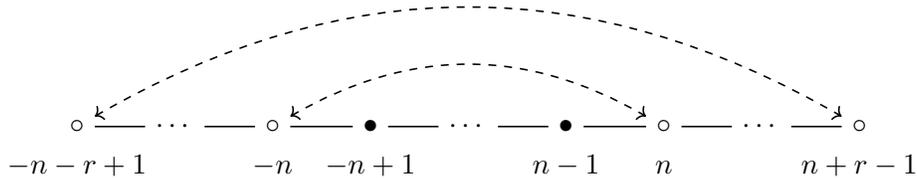
\begin{center}
\begin{tikzpicture}[scale=1, semithick]
\node (-4) [label=below:{$-n-r+1$}] at (0,0){$\circ$};
\node (-3)  at (1.3,0) {$\cdots$} ;
\node (-2) [label=below:{$-n$}] at (2.6,0){$\circ$};
\node (-1) [label=below:{$-n+1$}] at (3.9,0){$\bu$};
\node (0)  at (5.2,0){$\cdots$};
\node (1) [label=below:{$n-1$}] at (6.5,0){$\bu$};
\node (2) [label=below:{$n$}] at (7.8,0){$\circ$};
\node (3)  at (9.1,0){$\cdots$};
\node (4) [label=below:{$n+r-1$}] at (10.4,0){$\circ$};
\path (-4) edge (-3)
          (-3) edge (-2)
          (-2) edge (-1)
          (-1) edge (0)
          (0) edge (1)
          (1) edge (2)
          (2) edge (3)
          (3) edge (4);
\path (-4) edge[dashed,bend left,<->] (4)
          (-2) edge[dashed,bend left,<->] (2);
\end{tikzpicture}
\end{center}
(In case $n=0$, the black nodes are dropped; the nodes $n$ and $-n$ are identified and fixed by $\tau$.) The involution $\tau$ on $I$ sends $i \mapsto \tau(i)= -i$, for all $i$, and it induces an involution of $\U$, denoted again by $\tau$, by permuting the indices of its generators $E_i, F_i, K_i^{\pm 1}$.

Let \[
I_\bu =[1-n, n-1]
\]
be the set of all black nodes in $I$ so that
\[
I =I_\bu \cup I_\circ, \qquad \text{ where }  I_\circ :=I \backslash I_\bu.
\] 
Denote by $w_\bu$ the longest element in the Weyl group of the Levi subalgebra associated to $I_\bu$. Following \cite{BW18b}, we define 
\begin{align}
\label{eq:L}
\begin{split}
X_\imath &= X\big /\{\mu+ w_\bu \tau (\mu) \mid \mu \in X \},
\\
Y^\imath &=  \{ \nu- w_\bu \tau (\nu) \mid \nu \in Y\}.
\end{split}
\end{align}
We call an element in $X^\imath$ an $\imath$-weight and $X^\imath$ the $\imath$-weight lattice.

The $\imath$quantum group of type AIII, denoted by $\Ui$, depends on the parameters $\varsigma_i\in \Q(q)$, for $i\in I_\circ$, which satisfy the conditions $\va_i=\va_{-i}$, for $i\in I_\circ \backslash \{\pm  n \}$ \cite{Let02} (also cf. \cite{BK15, BW21}).
More precisely, $\Ui$ is the $\Q(q)$-subalgebra of $\U$ generated by $K_\mu \ (\mu\in Y^\io),\ E_i \ (i\in I_\bu)$, and
\begin{align}
\label{eq:Bi}
B_i= 
F_i+\va_iT_{w_\bu}(E_{\tau(i)}) K_i^{-1},\ \ \text{for } i\in I_\circ.
\end{align}
(In case $n=0$, $B_0$ will be allowed to take a more general form $B_0 =F_0 +\va_0 E_0 K_0^{-1} +\kappa_0 K_0^{-1}$, for an additional parameter $\kappa_0 \in \Q(q)$.)

Then $(\U, \Ui)$ forms a quantum symmetric pair of type AIII \cite{Let99, Let02} (cf. \cite{BW18a, BK19}). The algebra $\Ui$ satisfies the relations
\begin{align*}
K_\mu B_i &= q^{-(\mu,\alpha_i)} B_iK_\mu,\ \forall i\in I_\circ,  \\
K_\mu F_i &= q^{-(\mu,\alpha_i)} F_iK_\mu, \quad
K_\mu E_i = q^{(\mu,\alpha_i)} E_iK_\mu,\ \forall i\in I_\bu,  \mu \in Y^\io,
\end{align*}
and additional Serre type relations (which we shall not use explicitly in this paper). 

\subsection{$\io$Schur duality} 

In this subsection we will construct an $\io$Schur duality between type B Hecke algebra with two parameters $p,q$ and $\Ui$. To avoid considering a field extension of $\Q(q)$, we shall assume $p\in \Q(q)$. Then $\HB$ is a $\Q(q)$-algebra. The $\Q(q)$-vector space $\V=\oplus_{a\in \Ibw}\Q(q)v_a$ from \eqref{eq:V} can be identified with the natural representation of $\U$, where 
\begin{align}
  \label{eq:natural}
  \begin{split}
E_i v_a = \delta_{i+1,a} v_{a-1}, \qquad
 F_i v_a &= \delta_{i,a} v_{a+1}, 
\\
K_a v_a =q v_a, \quad
 K_a v_{a+1} &=q^{-1} v_{a+1}, \quad
K_a v_b =v_b \  (b\neq a, a+1).
\end{split}
\end{align} 


The tensor product $\V^{\otimes d}$ is naturally a $\U$-module via the comultiplication $\Delta$. 
Recall $\V^{\otimes d}$ is a right $\Hy_{B_d}$-module (and hence a right $\Hy_{S_d}$-module) from Lemma~\ref{lem:HB}. 

\begin{proposition}\label{prop:Jimbo} \cite{Jim86}
The actions of $\U$ and $\Hy_{S_d}$ on $\V^{\otimes d}$ commute with each other, and their images in $\End (\V^{\otimes d})$ form double centralizers.
\end{proposition}

We shall compute explicitly the action of $B_i$, for $i\in I_\circ$, on $\V$ in the following 2 lemmas. Recall $m=2n \in \N$. 

\begin{lemma}\label{lem:Tw} 
For $a\in \Ibw$ and $i \in I_\circ =[1-n-r, -n] \cup [n, n+r-1]$, we have
\[
T_{w_\bu}(E_{\tau(i)})(v_a)=\left\{\begin{aligned}
&E_{-i}(v_a), &|i| > \pt;\\
&E_{-\pt+1}E_{-\pt+2}\cdots E_{\pt-1}E_{\pt}(v_a),&i=-\pt; \\
&(-1)^{\nb-1}q^{-\nb+1}E_{-\pt}E_{-\pt+1}\cdots E_{\pt-2}E_{\pt-1}(v_a),& i=\pt.
\end{aligned}\right.
\]
\end{lemma}

\begin{proof}
For $i<-\pt$ and $i>\pt$, we have $T_{w_\bu}(E_{\tau(i)})=E_{-i}$. 

Let $i=-\pt$. We choose the following reduced expression of $w_\bu$: 
$$
w_\bu=(s_{-\pt+1}s_{-\pt+2}\cdots s_{\pt-1}) (s_{-\pt+1}s_{-\pt+2}\cdots s_{\pt-2})\cdots(s_{-\pt+1}s_{-\pt+2})(s_{-\pt+1}).
$$
Thus we compute
\begin{align}
T_{w_\bu}(E_{\tau(-\pt)})(v_a)&=T_{s_{-\pt+1}}\cdots T_{s_{\pt-1}}(E_{\pt})(v_a)
 \label{eq:Tva} \\
&=T_{s_{-\pt+1}}\cdots T_{s_{\pt-2}} (E_{\pt-1}E_{\pt}-q^{-1}E_{\pt}E_{\pt-1})v_a
 \notag \\
&=T_{s_{-\pt+1}}\cdots T_{s_{\pt-2}}(E_{\pt-1})E_{\pt}(v_a) -q^{-1} T_{s_{-\pt+1}}\cdots T_{s_{\pt-2}} (E_{\pt}E_{\pt-1})v_a. 
\notag
\end{align}
The second term on the RHS \eqref{eq:Tva} vanishes since $T_w (E_{\pt}E_{\pt-1})v_a = z T_w (E_{\pt}E_{\pt-1} v_{w(a)})$, for some scalar $z$, and $E_{\pt}E_{\pt-1} v_{w(a)} =0$ by \eqref{eq:natural}, for any $w, a$. Thus we derive that 
$$
T_{w_\bu}(E_{\tau(-\pt)})(v_a)=T_{s_{-\pt+1}}\cdots T_{s_{\pt-1}}(E_{\pt})(v_a)
=T_{s_{-\pt+1}}\cdots T_{s_{\pt-2}}(E_{\pt-1})E_{\pt}(v_a).
$$
Hence by a simple induction on $n$ we obtain
$$T_{w_\bu}(E_{\tau(-\pt)})(v_a)=E_{-\pt+1}E_{-\pt+2}\cdots E_{\pt-1}E_{\pt}(v_a).$$

Similarly, using another reduced expression
\[
w_\bu=(s_{\pt-1}s_{\pt-2}\cdots s_{-\pt+1})\cdots(s_{\pt-1}s_{\pt-2})(s_{\pt-1}),
\]
we compute $T_{w_\bu}(E_{\tau(\pt)})(v_a)$ as follows:
\begin{align*}
T_{w_\bu}(E_{\tau(\pt)})(v_a)&=T_{s_{\pt-1}}\cdots T_{s_{-\pt+1}}(E_{-\pt})(v_a)\\
&= T_{s_{\pt-1}}\cdots T_{s_{-\pt+2}}(E_{-\pt+1}E_{-\pt}-q^{-1}E_{-\pt}E_{-\pt+1})v_a \\
&= -q^{-1}E_{-\pt}T_{s_{\pt-1}}\cdots  T_{s_{-\pt+2}}(E_{-\pt+1})(v_a).
\end{align*}
Again by induction on $n$, recalling $m=2n$ we have
$$T_{w_\bu}(E_{\tau(\pt)})(v_a)
= (-1)^{\nb-1}q^{-\nb+1} E_{-\pt}E_{-\pt+1} \cdots E_{\pt-2}E_{\pt-1}(v_a).
$$
The lemma is proved. 
\end{proof}

Lemma~\ref{lem:Tw} together with the formula for $B_i$ in \eqref{eq:Bi} immediate imply the following. 
\begin{lemma} 
 \label{lem:Bi}
Let $a\in \Ibw$ and $i \in I_\circ$. The action of $B_i$ on $\V$ is given by:
$$
B_{-\pt}(v_a)=
\begin{cases}
v_{-\pt+\frac{1}{2}}, & \text{ if } a=-\pt-\frac{1}{2};\\
\va_{-\pt}v_{-\pt+\frac{1}{2}}, & \text{ if } a=\pt+\frac{1}{2};\\
0, &else,
\end{cases}
$$
$$
B_i(v_a)=
\begin{cases}
v_{i+\frac{1}{2}}, & \text{ if } a=i-\frac{1}{2};\\
\va_iv_{-i-\frac{1}{2}}, & \text{ if } a=-i+\frac{1}{2};\\
0, &else,\\
\end{cases}
\qquad \text{ for } |i| > \pt,
$$
and (recall $m=2n$)
$$
B_{\pt}(v_a)=
\begin{cases}
v_{\pt+\frac{1}{2}}+(-1)^{\nb-1}q^{-\nb}\va_{\pt}v_{-\pt-\frac{1}{2}}, & \text{ if } a=\pt-\frac{1}{2};
 \\
0,  &else.
\end{cases}
$$
\end{lemma}

From now on, we shall fix the parameters to be 
 \begin{align}
   \label{eq:para}
\left\{ 
\begin{aligned}
\va_i &=1,\ \text{ if } i\neq \pm\pt, \\
\va_{-\pt} &=p,  \qquad\qquad\qquad\qquad\qquad \text{ if } m=2n \in \Z_{\ge 1}, \\
 \va_{\pt} &=(-1)^{\nb-1}q^{\nb}p^{-1},  
 \end{aligned}
\right.
\end{align}
and
 \begin{align}
   \label{eq:para0}
\left\{ 
\begin{aligned}
\va_i &=1,\ \text{ if } i\neq 0, \\
\va_{0} &=q^{-1},  \qquad\qquad\qquad\qquad\qquad \text{ if } m=0.\\
 \kappa_0 &= \frac{p-p^{-1}}{q-q^{-1}}, 
 \end{aligned}
\right.
\end{align}
That is, for $m=0$, we take $B_0 =F_0 + q^{-1} E_0 K_0^{-1} +\frac{p-p^{-1}}{q-q^{-1}} K_0^{-1}$, following \cite{BWW18}.  

Introduce the $\Q(q)$-subspaces of $\V$:
\begin{align*}
\V_- &=\bigoplus_{a\in \Iwr}\Q(q)(v_{a} -pv_{-a}),\qquad
 \V_\bu =\bigoplus_{a\in \I_\bu}\Q(q) v_a,\\
 \V_+ &=\bigoplus_{a\in \Iwr}\Q(q)(v_{a}+p^{-1}v_{-a}).
\end{align*}

\begin{lemma}
  \label{lem:Vpm}
Assume \eqref{eq:para}--\eqref{eq:para0}. 
Then $\V_-$ and $\V_\bu \oplus \V_+$ are $\Ui$-submodules of $\V$. 
Hence, we have a $\Ui$-module decomposition $\V= (\V_\bu\oplus \V_+) \oplus \V_-$.
\end{lemma}

\begin{proof}
Follows by a direct computation using the formulas \eqref{eq:natural} and Lemma~\ref{lem:Bi}. 
\end{proof}

The decomposition of $\V$ above is also compatible with the $H_0$-action. 

\begin{lemma}
 \label{lem:H0}
The Hecke generator $H_0$ acts on $\V_-$ as $(-p^{-1} ) \text{Id}$ and acts on $\V_\bu \oplus \V_+$ as $p\cdot \text{Id}$.
\end{lemma}

\begin{proof}
Follows by Lemma~\ref{lem:HB}. 
\end{proof}

\begin{theorem}\label{thm:UiHB}
Suppose the parameters satisfy \eqref{eq:para}--\eqref{eq:para0}. 
Then the actions of $\Ui$ and $\HB$ on $\V^{\otimes d}$ commutes with each other: 
\[
\Ui \stackrel{\Psi}{\curvearrowright} \V^{\otimes d} \stackrel{\Phi}{\curvearrowleft} \Hy_{B_d}. 
\]
Moreover, $\Psi(\Ui)$ and $\Phi(\HB)$ form double centralizers in $\End (\V^{\otimes d})$.
\end{theorem}

\begin{proof}
As the case for $m=0$ was covered in \cite{BWW18}, we shall assume $m\ge 1$ below. 

By the Jimbo duality (see Proposition~\ref{prop:Jimbo}), we know that the action of $\U$ commutes with the action of $H_i$, for $1\le i \le d-1$. Thus, to show the commuting actions of $\Ui$ and $\HB$, it remains to check the commutativity of the actions of $H_0$ and the generators of $\Ui$.  

To that end, it suffices to consider $d=1$ (thanks to the coideal property of $\Ui$ and the fact that the action of $H_0$ depends solely on the first tensor factor). In this case, the commutativity between $\Ui$-action and $H_0$-action on $\V$ follows directly from Lemmas~\ref{lem:Vpm} and \ref{lem:H0}. 
%
%
%
%
%
%

The double centralizer property is equivalent to a multiplicity-free decomposition of $\V^{\otimes d}$ as an $\Ui\otimes \Hy_{B_d}$-module, which reduces by a deformation argument to the $q=1$ setting. At the specialization $q\mapsto 1$, $\Ui$ becomes the enveloping algebra of $\mathfrak{sl}(r+m) \oplus \mathfrak{gl}(r)$, $\V =(\V_\bu \oplus \V_+) \oplus \V_-$ becomes the natural representation of  $\mathfrak{sl}(r+m) \oplus \mathfrak{gl}(r)$, on which $s_0\in W_d$ acts as $(\text{Id}_{\V_\bu \oplus \V_+}, -\text{Id}_{\V_-})$. The multiplicity-free decomposition of $\V^{\otimes d}$ at $q=1$ can be established by a standard approach  where the simples are parameterized by ordered pairs of partitions $(\lambda, \mu)$ such that $l(\lambda) \le r+m, l(\mu) \le r$ and $|\lambda| +|\mu| =d$. 
\end{proof}

\begin{remark} 
  \label{rem2.2} 
Theorem~\ref{thm:UiHB}  is a common generalization of $q$-Schur dualities of type A and B. It specializes to Jimbo duality (Proposition~\ref{prop:Jimbo}) when $r=0$. (In this case, $\Ui=\U$, and $H_0$ acts as $p\cdot \text{Id}$ and so the action of $\HB$ reduces to the action of $\Hy_{S_d}$.) 

On the other hand, for $m=0,1$, Theorem~\ref{thm:UiHB} reduces to \cite[Theorems~ 5.4, 6.27]{BW18a} (for $p=q$), \cite[Theorem~3.4]{Bao17} (for $p=1$), and \cite[Theorems~ 2.6, 4.4]{BWW18} for general $p$. 
The conventions {\em loc. cit.} are consistent with each other, while a different comultiplication for $\U$ is used in this paper; this has led to a different partial ordering on $\Ibw^d$ and a switch of $q,p$ from {\em loc. cit.} to $q^{-1},p^{-1}$ for the action of Hecke algebra; cf. Lemma~\ref{lem:HB}. 
\end{remark}

\section{$\io$Canonical basis on the tensor module}
 \label{sec:iCB}

In this section, we fix the parameters $\va_i$ ($i\in I_\circ$) as in \eqref{eq:para}--\eqref{eq:para0} as for Theorem~\ref{thm:UiHB}, and further assume that $p \in q^\Z$. We show that the bar involution on the tensor space is compatible with the bar involutions on the algebras $\Ui$ and $\HB$. 
We further show that the $\imath$canonical bases on the tensor space arising from the $\imath$quantum group and from Hecke algebra coincide. 

\subsection{Generalities of $\io$canonical bases}

In this subsection we review several constructions in the theory of $\io$canonical basis \cite{BW18a, BW18b}. 

A bar involution $\iba$ on $\Ui$ was given in \cite{BW18a} of the quasi-split type AIII (i.e., $m=0, 1$); it was stated therein that a bar involution exists for general $\imath$quantum groups, and this was subsequently established in \cite{BK15}. In any case, the existence of the bar involution for $\Ui$ of type AIII under the assumption on parameters \eqref{eq:para}--\eqref{eq:para0} can be checked directly from the known presentation of $\Ui$.

\begin{lemma} 
 \label{lem:iba}
There is a unique bar involution on $\Ui$, denoted by $\psi_{\io}$, such that
$$\iba(q)=q^{-1},\ \iba(B_j)=B_j,\ \iba(E_i)=E_i,\ \iba(F_i)=F_i,\ \iba(K_\mu)=K_{-\mu},
$$
for $j\in I_\circ, i\in I_{\bu},$ and $\mu \in Y^\imath$. 
\end{lemma}

Note that  $\iba(p)=p^{-1}$ as $p\in q^\Z$. The two bar maps on $\Ui$ and $\U$ are not compatible under the inclusion map $\Ui\to \U$. As a generalization of quasi R-matrix \cite[4.1.2]{Lus93}, a notion of quasi K-matrix (also known earlier as intertwiner), denoted by $\Upsilon$,  was formulated in \cite{BW18a}; a proof in greater generality was subsequently given in \cite{BK19}; also cf. \cite{BW18b}. 

\begin{proposition} \cite{BW18a, BK19, BW18b}
  \label{prop:upsilon}
There exists a unique family of elements $\Upsilon_\mu\in \U^+_{\mu}$, such that $\Upsilon_0=1$ and $\Upsilon=\sum_\mu \up_\mu$ satisfies 
\[
 \iba(u)\up=\up\psi(u),\quad \forall u \in \Ui.
 \]
Moreover, $\up_\mu=0$ unless $w_\bu \tau (\mu) =\mu$.
\end{proposition}

Given based $\U$-modules $M_i$ ($i=1,2$) with bar involution $\bar{\,\,}$, Lusztig \cite[27.3.1]{Lus93} defined a bar involution on $\psi: M_1 \otimes M_2 \rightarrow M_1 \otimes M_2$ by $\psi(x_1\otimes x_2) =\Theta (\bar x_1 \otimes \bar x_2)$, where $\Theta$ is the quasi-R matrix. The natural representation $\V$ of $\U$ admits a bar involution such that $\bar v_i =v_i$, for all $i$. Inductively, we obtain a bar involution $\psi$ on $\V^{\otimes d}$. 

The $\U$-weight of $f \in \Ibw^d$ is defined to be $\text{wt} (f)=\sum_{i=1}^{d}\epsilon_{f(i)}$. Recall the $\imath$weight lattice $X_\imath$ from \eqref{eq:L}. Define the $\Ui$-weight of $f$ to be
$$
\text{wt}_\io(f)=\sum_{i=1}^d \bar \epsilon_{f(a)} \in X_\imath,
$$
which is the image of $\text{wt} (f)$ in $X_\imath$. Following \cite[(5.2)]{BW18b} we define the following partial order $\preceq_\io$ on $\Ibw^d$:
\begin{align}
  \label{eq:order}
g\preceq_\io f \Leftrightarrow \text{wt}_\io(g)=\text{wt}_\io(f) \text{ and } \text{wt} (g)- \text{wt} (f)\in \N [I] \cap \N[w_\bu I].
\end{align}
We also write $g\prec_\io f$ if $g\preceq_\io f$ and $g\neq f$. 
A $\Ui$-module $M$ equipped with a bar involution $\iba$ is called {\em $\imath$-involutive} if
\[
\iba(uz) =\iba(u)\iba(z),\quad \forall u\in \Ui,z\in M.
\]

\begin{proposition}  \cite{BW18b}
 \label{prop:ibar}
The $\U$-module $\V^{\otimes d}$ is an $\imath$-involutive $\Ui$-module with the bar involution
\begin{equation}
 \label{eq:bar1}
 \iba:=\up\circ \psi.
\end{equation}
Moreover, for $f \in \Ibw^d$, we have
\begin{align}  \label{order}
\up(M_f) \in M_f+\sum_{g\prec_\io f}  \Z[q,q^{-1}] M_g. 
\end{align}
\end{proposition} 

\begin{proof}
The first statement is a special case of \cite[Proposition 5.1]{BW18b}. 
The formula \eqref{order} follows by Proposition~\ref{prop:upsilon} and the definition of the partial order $\preceq_\io$ in \eqref{eq:order}. 
\end{proof}

Below is a very special case of \cite[Theorem 5.7]{BW18b} concerning about $\V^{\otimes d}$. 

\begin{proposition}
 \label{prop:iCB}
(1) The $\Ui$-module $\V^{\otimes d}$ admits a unique $\io$canonical basis $\{C_g | g \in \Ibw^d\}$ which is characterized by 2 properties: (i) $C_g$ is $\iba$-invariant; (ii) $C_g$ is of the form:
\begin{align}
  \label{eq:bf}
C_g \in M_g+\sum_{g' \in \Ibw^d} q^{-1}\Z[q^{-1}]M_{g'}.
\end{align}
(2) The $\V^{\otimes d}$ admits a unique dual $\io$canonical basis $\{C^*_g | g \in \Ibw^d\}$ such that (i) $C^*_g$ is $\iba$-invariant; (ii) $C^*_g \in M_g +\sum_{g' \in \Ibw^d} q \Z[q]M_{g'}.$
\end{proposition}
It was then shown that the $C_g$ satisfy a stronger property: 
$
C_g \in M_f+\sum_{g'\prec_\io g} q^{-1}\Z[q^{-1}]M_{g'}.
$

\subsection{$\imath$Canonical basis on $\V$}

Recall the notations $\Iwl, \Iwr, \Ib$ from \eqref{eq:III} and $m=2n$.

\begin{lemma}
  \label{lem:a}
We have
\begin{align}
\iba (v_a) =\up(v_a) &= v_a,\quad a\in \Iwl\cup \Ib;
\label{iba1}
 \\
\iba (v_a) =\up(v_a) &= v_a+(p^{-1}-p)v_{-a},\quad a\in \Iwr.
\label{iba2}
\end{align}
\end{lemma}

\begin{proof}
As $v_a$ is bar invariant (i.e., $\psi$-invariant), the equality $\iba (v_a) =\up(v_a)$, for all $a$, follows by definition $\psi_i =\Upsilon \psi$ in \eqref{eq:bar1}. 

Let $a\in \Iwl\cup \Ib$. The equality $\up(v_a) =v_a$  is a direct consequence of \eqref{order}. 

It remains to prove the formula \eqref{iba2}, for $a \in \Iwr$ (i.e., $a\in \Ibw$ with $a\ge n+\frac12$). By a simple induction on $a$, we have 
\begin{align}  \label{eq:BBB}
B_{a-\frac12} \cdots B_{\pt+1}B_{\pt}(v_{\pt-\frac{1}{2}})=v_{a}+p^{-1}v_{-a}.
\end{align}
The element \eqref{eq:BBB} is $\iba$-invariant, since the $B_k$'s are $\iba$-invariant by Lemma~\ref{lem:iba}, $v_{\pt-\frac{1}{2}}$ is $\iba$-invariant by \eqref{iba1}, and $\V$ is $\imath$-involutive by Proposition~\ref{prop:ibar}. On the other hand, thanks to $-a \in \Iwl$, we have $v_{-a}$ is $\iba$-invariant by \eqref{iba1}. Hence, it follows that 
\begin{align*}
\iba(v_a) &= \iba \big( (v_{a}+p^{-1}v_{-a}) - p^{-1} v_{-a} \big)
\\
&= (v_{a}+p^{-1}v_{-a}) - p v_{-a}
\\
&= v_{a}+(p^{-1}-p)v_{-a}.
\end{align*}
This proves the lemma. 
\end{proof}

\begin{proposition} 
The $\io$canonical basis of $\V$ is given by 
\begin{enumerate}
\item
$\{v_a\mid a\in \Iwl \cup \Ib\}\cup\{v_a+p^{-1}v_{-a},\ a\in\Iwr\}$, if $p=q^{\Z_{>0}}$;
\item
$\{v_a\mid a\in \Ibw \}$, if $p=1$; 
\item
$\{v_a\mid a\in \Iwl \cup \Ib\}\cup\{v_a-pv_{-a},\ a\in\Iwr\}$, if $p=q^{\Z_{<0}}$.
\end{enumerate}
\end{proposition}

\begin{proof}
It follows by Lemma~\ref{lem:a} that these elements are $\iba$-invariant, and they are clearly of the form \eqref{eq:bf}. Hence the proposition follows by the characterization of $\imath$canonical basis in Proposition~\ref{prop:iCB}.
\end{proof}

\subsection{Compatible bar involutions and canonical bases}

We formulate a compatibility between several  bar involutions, which generalizes \cite[Theorem 5.8]{BW18a}; the same proof therein carries over. 

\begin{proposition}
  \label{prop:ioba}
There exists a unique anti-linear bar involution
$ \iba\colon \V^{\otimes d} \rightarrow \V^{\otimes d}$ such that $\iba(M_f) =M_f$, for $f\in \Ianti$, and it  
is compatible with the bar involutions on $\Hy_{B_d}$ and $\Ui$; that is, for $u\in \Ui$, $v\in \V^{\otimes d}$, and $h \in \HB$,  
$$
\iba(uvh)=\iba(u)\iba(v)\bar h.
$$
\end{proposition}

\begin{remark}
Thanks to the compatibility with the bar map on $\HB$ and $\overline{M}_f=M_f$, the bar map $\iba$ on $\V^{\otimes d}$ when restricted to $\M_f$, for anti-dominant $f$, coincides with $\iba$ in Proposition~\ref{prop:iHba}. 
\end{remark}

Recall from \eqref{eq:decomp} that $\V^{\otimes d}$ is a direct sum of the quasi-permutation modules $\M_f$ of $\HB$. The union of the (dual) quasi-parabolic KL bases on the direct summands $\M_f$ (see Theorem~\ref{thm:CBMf} and Proposition~\ref{prop:dualCBMf}) provide us a (dual) KL  basis on $\V^{\otimes d}$.

\begin{theorem}
  \label{thm:iCBsame}
The (dual) $\imath$canonical bases on $\V^{\otimes d}$ (viewed as a $\Ui$-module) coincides with the (dual) KL bases on $\V^{\otimes d} =\oplus_{f} \M_f$ (viewed as an $\HB$-module). 
More precisely, we have the identifications of bases in $\M_f$: $C_{f\cdot \sigma} =C_\sigma$ and $C^*_{f\cdot \sigma} =C^*_\sigma$, for $f\in \Ianti$ and $\sigma \in {}^fW$. 

(See Theorem~\ref{thm:CBMf}, Proposition~\ref{prop:dualCBMf} and Proposition~\ref{prop:iCB} for notations.)
\end{theorem}

\begin{proof}
We only need to consider the $\imath$canonical basis as the dual version follows by the same argument. 
Both bases are invariant under the same bar map $\iba$ (thanks to Proposition~\ref{prop:ioba}) and are of the form $C_g \in M_g +\sum_{g'\in \Ibw^d} q^{-1} \Z[q^{-1}] M_{g'}$. Now by the uniqueness in Proposition~\ref{prop:iCB} the $\imath$canoical basis coincides with the KL basis. The precise formula $C_{f\cdot \sigma} =C_\sigma$ follows as both sides have the same leading term $M_{f\cdot \sigma}$. 
\end{proof}

\begin{remark}
{\quad}
\begin{enumerate}
\item
In case $m=0$ (the case $m=1$ is similar), Proposition~\ref{prop:ioba} and Theorem~\ref{thm:iCBsame} reduce to \cite[Theorem~5.8, Remark 5.9]{BW18b} and \cite[Proposition~3.9, Theorem~3.10]{BWW18}. Here we choose not to use general weight functions as in \cite{BWW18} to avoid clumsy notations thought there is no difficulty in setting up in such a generality. 
\item
In case $r=0$, the $\imath$Schur duality reduces to Jimbo duality by Remark~ \ref{rem2.2}.
Accordingly Proposition~\ref{prop:ioba} and Theorem~\ref{thm:iCBsame} recover the main results in \cite{FKK98}. 
\item
The $\imath$canonical basis on $\V_\bu^{\otimes d}$ coincides with Lusztig's canonical basis. 
By Theorem~\ref{thm:iCBsame} and Example~\ref{ex:BA}, parts of the $\imath$canonical basis on $\V^{\otimes d}$ can be identified with (parabolic) Kazhdan-Lusztig bases of type A or type B, but not always. 
\end{enumerate}
\end{remark}

\subsection{Realizing $H_0$ via $K$-matrix}

For quantum symmetric pair $(\U, \Ui)$ of quasi-split type AIII, an $\Ui$-module isomorphism $\mathcal T$ on any weight  $\U$-module $M$ was constructed \cite[Theorem 2.18]{BW18a} by twisting the quasi K-matrix $\Upsilon$ by a weight function $\xi: X \rightarrow \C$. This construction has been generalized to general quantum symmetric pairs \cite[Corollary 7.7]{BK19}, who referred to it as a $K$-matrix and changed the notation to be $\mathcal K$. Let us quickly review it. 

Let $\gamma:\I\to \Q(q)$ be a function defined by
$$
\gamma(i)=\left\{\begin{aligned}
&1,\ \ &\text{ if } &i\in I_\bu \\
&-\va_i,\ \ &\text{ if } &i\in I_\circ.
\end{aligned}\right.
$$
Define a function $\xi\colon X \to \Q(q)$ by the following recursion: 
\begin{equation}
  \label{eq:recur}
\xi(\mu+\alpha_i)= \gamma(i) q^{(\alpha_i,w_\bu \tau (\alpha_i)) -(\mu,\alpha_i -w_\bu \tau (\alpha_i))}\xi(\mu),\ \ \forall \mu \in X, i\in \I.
\end{equation}
The function $\xi$ induces a linear map $\widetilde \xi$ on any weight module $M=\sum_{\mu\in X}M_\mu$ by letting 
$$
\widetilde\xi(z)=\xi(\lambda)z,\ \ \text{ for }z\in M_{\lambda}.
$$

From now on, we fix the function $\xi$ with $\xi({\epsilon_{\pt+r-\frac{1}{2}}})=1$. 

\begin{lemma} \label{lem:xi}
Let $\xi({\epsilon_{\pt+r-\frac{1}{2}}})=1$. Then we have 
$$
\xi(\epsilon_a)=\left\{ \begin{aligned}
&(-q)^{\pt+r-\frac{1}{2}-a}, 
&a\leq -\pt-{\frac{1}{2}}; \\
&(-q)^{\nb+r-1}  p^{-1},&-\pt+\frac{1}{2}\leq a \leq \pt+\frac{1}{2}; \\
&(-q)^{\pt+r-\frac{1}{2}-a}, &a\geq \pt+{\frac{3}{2}}.
\end{aligned}\right.
$$
\end{lemma}

\begin{proof}
The function $\xi$ is completely determined by the recursion  \eqref{eq:recur} and the fixed value for $\xi({\epsilon_{\pt+r-\frac{1}{2}}})$.
Note that $\xi(\epsilon_{a})=\xi(\epsilon_{a+1}+\alpha_{a+\frac{1}{2}})$.
Thus by \eqref{eq:recur}, for $a\leq -\pt-\frac{3}{2}$, we have 
$$
\xi(\epsilon_{a})
=\gamma(a+\frac{1}{2})q^{(\alpha_{a+\frac{1}{2}},w_\bu \tau (\alpha_{a+\frac{1}{2}}))-(\epsilon_{a+1},\alpha_{a+\frac{1}{2}} -w_\bu \tau (\alpha_{a+\frac{1}{2}}))}\xi(\epsilon_{{a+1}})=-q\xi(\epsilon_{{a+1}}).
$$
The remaining cases of the recursion can be similarly made explicit. 
\end{proof}

\begin{proposition} \cite[Theorem 2.18]{BW18a} \cite[Corollary 7.7]{BK19}
\label{prop:K}
For any finite dimensional $\U$-module $M$ and any $\xi$ which satisfies the recursion in \eqref{eq:recur}, the element $\mathcal K=\up \widetilde \xi T_{w_\bu}^{-1}T_{w_0}^{-1}$ defines an $\Ui$-module isomorphism:
$$\mathcal K \colon M\longrightarrow M,\qquad z\mapsto \up\circ \widetilde \xi \circ T_{w_\bu}^{-1}T_{w_0}^{-1}(z).$$
\end{proposition}

We compute the action of $\mathcal K$ on the natural $\U$-module $\V$.

\begin{lemma}
The $\Ui$-isomorphism $\mathcal K$ on $\V$ acts as $(-p) \text{Id}$ on the submodule $\V_-$ and as $p^{-1} \text{Id}$ on $\V_+ \oplus \V_\bu$.
\end{lemma}

\begin{proof}
First one computes that the actions of $T_{w_0}$ and $T_{w_\bu}$ on $\V$ are given by
\begin{align*}
T_{w_0}(v_a) &=(-q)^{r+\nb-a-\pt-\frac{1}{2}}v_{-a}, \quad \forall a\in \Ibw,\\
T_{w_\bu}(v_a) &=\left\{\begin{aligned}
&(-q)^{\nb-a-\pt-\frac{1}{2}}v_{-a},& \text{ if } a\in \I_\bu; \\
&v_a,&else.
\end{aligned}
\right.
\end{align*}

Hence by a direct computation using these 2 formulas and Lemma~\ref{lem:xi} we have 
\begin{align}
  \label{5.14}
\widetilde\xi \circ T_{w_\bu}^{-1}T_{w_0}^{-1}(v_a)=\left\{\begin{aligned}
&v_{-a},\ &a\in \Iwl\cup\Iwr; \\
&p^{-1}v_a,\ &a\in \Ib.
\end{aligned}\right.
\end{align}

By Lemma~\ref{lem:a} we have
\begin{align*}
\K (v_{\pt+\frac{1}{2}}-pv_{-\pt-\frac{1}{2}}) &=-p(v_{\pt+\frac{1}{2}}-pv_{-\pt-\frac{1}{2}}),\\ 
\K (v_{\pt+\frac{1}{2}}+p^{-1}v_{-\pt-\frac{1}{2}}) &=p^{-1}(v_{\pt+\frac{1}{2}}+p^{-1}v_{-\pt-\frac{1}{2}}).
\end{align*}
Again by Lemma~\ref{lem:a} we have $\K (v_a)=p^{-1} v_a,\ \forall a\in \I_\bu.$
Now the lemma follows.
\end{proof}

The action of the generators $H_i$ for $\Hy_{S_d}$, for $1\le i \le d-1$, on $\V^{\otimes d}$ are realized via R-matrix \cite{Jim86} (also see \cite{LW23}). This has the following generalization for the generator $H_0$ in $\HB$.

\begin{proposition}
  \label{prop:K2}
The action of $H_0^{-1}$ on $\V^{\otimes d}$ in Lemma~\ref{lem:HB} is realized via the $K$-matrix as $\K\otimes \text{Id}^{\otimes d-1}$.
\end{proposition}
In case $m=0$ or $1$, Proposition~\ref{prop:K2} is established in \cite{BW18a, BWW18}. The property of a K-matrix in Proposition~\ref{prop:K} also provides a conceptual explanation for the commutativity of $H_0$ and $\Ui$ acting on $\V^{\otimes d}$. 

\subsection{$\io$Schur algebra}

We formulate the $\io$Schur algebra arising from $\imath$Schur duality.

\begin{definition}
The $\io$Schur algebra $\sch(r| \nb|r,d)$ is defined to be 
$$\sch(r| \nb|r,d)= \End_{\HB}(\V^{\otimes d})=\Psi(\Ui).$$
(The second equality follows by the double centralizer property in Theorem~\ref{thm:UiHB}.)
\end{definition}
	
\begin{remark}
When $r=0$, our $\io$Schur algebra specializes to $q$-Schur algebra of type A \cite{DJ89}. 
When $\nb=0$ or $1$, our $\io$Schur algebra specializes to the quasi-split $\io$Schur algebra in \cite{Gr97, BW18a, BKLW18, LL21}.
\end{remark}

\begin{lemma}
There exists a unique (anti-linear) bar involution $~^-$ on $\sch(r| \nb|r,d)$ such that
$$\bar \rho(M_{g'}h)=\delta_{g,g'}\iba \big( \rho(M_{g'}) \big) h,\quad \forall h\in \Hy_{B_d},\ g' \in \Ianti,
$$
for any $\rho\in Hom_{\Hy_{B_d}}(\M_g,\M_f)\subset \sch(r| \nb|r,d)$, and any $f, g \in \Ianti$. 
\end{lemma}

\begin{proof}
We first check that the map $\bar \rho$  is well defined. Indeed, 
\[
\iba \big( \rho(M_{g'}) \big) h
=\iba \big( \rho(M_{g'}) \bar h \big) 
=\iba \big( \rho(M_{g'} \bar h) \big) 
=\iba \big( \rho(\iba(M_{g'} h) ) \big).
\]
The last expression above depend on $M_{g'} h$ (not just $h$), and so $\bar \rho$ is well defined. By this last expression it is also clear that $~^-$ on $\sch(r| \nb|r,d)$ is anti-linear and it is an involution. 
\end{proof}

\begin{remark}
The $\io$Schur algebras $\sch(r| \nb,d)$ are Morita equivalent to (but not isomorphic to) various versions of $(Q,q)$-Schur (or $q$-Schur$^2$) algebras studied in~\cite{DJM98} and~\cite{DS00}. The $\HB$-module $\V^{\otimes d}$ is a direct sum of quasi-permutation modules somewhat different  from those considered {\em loc. cit.}, but the results {\em loc. cit.} can be used to provide a basis for  $\sch(r| \nb,d)$.

The current work leads to the natural question of establishing a canonical basis for the $\io$Schur algebra $\sch(r| \nb,d)$ and developing its connection to the $\imath$canonical basis on the modified $\imath$quantum group $\dot \U^\imath$.
\end{remark}

\section{An inversion formula for quasi-parabolic KL polynomials}
  \label{sec:inversion}
  
In this section we prove an inversion formula for quasi-parabolic KL polynomials, generalizing \cite{KL79} and \cite{Do90}; also cf. \cite{So97}. Inspired by the type A works \cite{Br06} and \cite{CL16}, our approach is based on the tensor module formulation and uses the $\imath$Schur duality. 

\subsection{Symmetries $\varrho$, $\sigma_\imath'$ and $\sigma_\imath$}

Let $(\cdot,\cdot)$ denote the standard symmetric bilinear form on $\V^{\otimes d}$ defined by
\begin{equation}
 \label{eq:MMBF}
(M_f,M_g)=\delta_{f,g},\ \forall f,g\in \Ibw^d.
\end{equation}

We recall several symmetries of $\U$; cf. \cite{Lus93}.
\begin{lemma}
(1) There is an anti-involution $\varrho$ of $\U$ such that, for $i\in I, \mu \in Y$, 
\begin{equation}
\varrho(E_i)=q^{-1}F_iK_i,\quad \varrho(F_i)=q^{-1}E_iK_i^{-1},\quad  \varrho(K_\mu)=K_\mu.
\end{equation}
(2) There is an anti-involution $\sigma$ of $\U$ such that, for $i\in I, \mu \in Y$, 
\begin{equation}
\sigma(E_i)=E_{i},\quad  \sigma(F_i)=F_{i},\quad  \sigma(K_\mu)=K_{-\mu}.
\end{equation}
\end{lemma}
The bilinear form $(\cdot,\cdot)$ on $\V^{\otimes d}$ defined by
\eqref{eq:MMBF} satisfies  (cf. \cite{Lus93})
\begin{equation}
\label{bisym-equa}
(u x, y)=(x, \varrho(u)y),
\end{equation}
 for all $x, y \in \V^{\otimes d},$ and $u \in \U$.
 
 Following \cite[\S 3.6.2]{BW21}, we consider an anti-linear anti-involution $\sigma_\imath'$ of $\U$ such that
\begin{align}  \label{eq:isigma1}
\sigma_\io' =\sigma \circ \tau \circ \psi. 
\end{align}
Note the (anti-)involutions $\sigma, \tau,$ and $\psi$ commute with each other. 
\begin{lemma}
\label{lem:DD}
The maps $\sigma_\io'$ and $\varrho$ are coalgebra morphisms, that is, 
\begin{align*}
(\sigma_\io'\otimes \sigma_\imath')\Delta(u) &=\Delta(\sigma_\imath'(u)),
\\
(\varrho \otimes \varrho)\Delta(u) &=\Delta(\varrho(u)),
\quad \text{ for all } u\in \U.
\end{align*}
\end{lemma}

\begin{proof}
It is straightforward to check on generators $u\in \U$ that
\begin{align*}
(\sigma \psi \otimes \sigma \psi)\Delta(u) &=\Delta(\sigma \psi(u)), 
\\
(\tau \otimes \tau)\Delta(u) &=\Delta(\tau (u)).
\end{align*}
Hence these 2 identities hold for all $u\in \U$ since $\sigma\psi$ and $\tau$ are (anti-)involutions on $\U$. The lemma now follows from by definition of $\sigma_\io' =\sigma \psi \tau$ in \eqref{eq:isigma1} and these identities.

The (well known) statement that $\varrho$ is a coalgebra morphism (cf. \cite{CL16}) can also be checked on the generators of $\U$ directly. 
\end{proof}

By the proof of \cite[Proposition 3.13]{BW21}, $\sigma_\io'$ defined in \eqref{eq:isigma1} preserves the subalgebra $\Ui$ of $\U$. Note that $\iba$ and $\sigma_\io'$ commute on $\Ui$.
\begin{lemma} \cite[Proposition 3.13]{BW21}
 \label{lem:invol2} 
We have an anti-linear anti-involution $\sigma_\io'$ of $\Ui$ by restriction and 
a $\Q(q)$-linear anti-involution $\sigma_\io$ of $\Ui$ given by
\begin{align}  \label{eq:isigma}
\sigma_\imath=\iba\circ \sigma_\io'. 
\end{align}
\end{lemma}

\subsection{Quasi R-matrix $\Theta^\io$}

Recall the quasi K-matrix $\up$ from Proposition~\ref{prop:upsilon}. As in \cite[(3.1)]{BW18a}, we define the quasi R-matrix $\Theta^\io$ associated to the quantum symmetric pair $(\U, \Ui)$ by 
\begin{align*}
\Theta^\io &=\Delta(\up)\Theta(\up^{-1}\otimes 1).
\end{align*}
We also define
\begin{align}
\label{eq:iDelta}
\begin{split}
\bD: & \Ui\longrightarrow \Ui\otimes \U, \\
\bD(u) &=(\iba\otimes \psi)\Delta(\iba(u)),\ \forall u\in \Ui.
\end{split}
\end{align}

The fundamental properties of $\Theta^\io$ in Proposition~\ref{prop:iTheta}~ (1)-(2)  below were established in \cite[Propositions 3.2, 3.5]{BW18a} and generalized in \cite[Propositions 3.9-3.10]{Ko20}. The uniqueness below can be found in the proof of \cite[Propositions 3.7]{BW18a}, and in general can be derived from a variant of the interwining property given by \cite[(3.28)]{Ko20}.

\begin{proposition} (cf. \cite{BW18a, Ko20}) 
\label{prop:iTheta}
\quad
\begin{enumerate}
\item
We have $\Theta^\io =\sum_{\mu\in \N I} \Theta^\io_\mu$, where $\Theta^\io \in \Ui \otimes \U_\mu^+$ and $\Theta^\io_0=1\otimes 1$. 

\item
$\Theta^\io$ satisfies that 
$
\Delta(u)\Theta^\io=\Theta^\io\bD(u).
$ 
\end{enumerate}
Moreover, an element $\Theta^\io$ of the form (1) satisfying the intertwining property (2) is unique.
\end{proposition}
 

The following new property of $\Theta^\io$ is actually valid for a general quantum symmetric pair as in \cite{BW21}. It will play a role in the proof of Theorem~\ref{thm:bisym} below.

\begin{lemma}
\label{lem:iTheta2}
We have
$(\sigma_\io\otimes \sigma \tau)(\Theta^\io)=\Theta^\io$.
\end{lemma}

\begin{proof}
Denote $\check\Theta^\io =(\sigma_\io\otimes \sigma \tau) (\Theta^\io)$, which is well defined thanks to Lemma~\ref{lem:invol2} and Proposition~\ref{prop:iTheta}(1).

Applying the {\em anti}-involution $\sigma_\io\otimes \sigma \tau$ to the identity $\Delta(u)\Theta^\io=\Theta^\io\bD(u)$ (see Proposition~\ref{prop:iTheta}), we obtain
\begin{align*}
\check\Theta^\io \ (\sigma_\io\otimes \sigma \tau) \Delta(u)  
&= (\sigma_\io\otimes \sigma \tau) \bD(u) \ \check\Theta^\io,
\end{align*}
which can be rewritten as 
\begin{align*}
\check\Theta^\io \ (\iba \otimes \psi) (\sigma_\io' \otimes \sigma_\io') \Delta(u)  
&= (\sigma_\io' \otimes \sigma_\io') \Delta (\iba (u)) \ \check\Theta^\io.
\end{align*}
Applying Lemma~\ref{lem:DD} to the above identity, we obtain
\begin{align*}
\check\Theta^\io \ (\iba \otimes \psi)  \Delta(\sigma_\io'(u) )  
&= \Delta (\sigma_\io' \iba (u)) \ \check\Theta^\io. 
\end{align*}
Setting $x= \sigma_\io' \iba (u)= \iba \sigma_\io' (u)$, the above identity can be read in the notation of \eqref{eq:iDelta} as 
\begin{align*}
\check\Theta^\io \ \bD(x)  
&= \Delta (x) \ \check\Theta^\io,
\end{align*}
that is, $\check\Theta^\io$ satisfies the intertwining property in Proposition~\ref{prop:iTheta}(2).
Clearly, $\check\Theta^\io$ also satisfies Proposition~\ref{prop:iTheta}(1). 
It follows by the uniqueness (see Proposition~\ref{prop:iTheta})  that 
$\check\Theta^\io =\Theta^\io.$
\end{proof}

\subsection{A bilinear form $\la \cdot, \cdot \ra$}

We introduce an {\em anti}-linear map  
\begin{align}
\D\colon & \V^{\otimes d} \longrightarrow \V^{\otimes d},
 \label{eq:D}
\\
 & \D(M_f) =M_{-f}, \text{ for }  f\in \Ibw^d.
 \notag
\end{align}
We define a new bilinear form $\la\cdot,\cdot\ra$ on $\V^{\otimes d}$ in terms of the standard one $(\cdot,\cdot)$ in \eqref{eq:MMBF} by letting 
\begin{equation}
 \label{eq:BF}
\langle x,y \rangle:=(x,\D\circ\iba(y)), 
\ \forall x, y \in \V^{\otimes d}.
\end{equation}

The following lemma will also be used in the proof of Theorem \ref{thm:bisym}.
\begin{lemma}
\label{bisym-lemma2} 
For all $x \in \V^{\otimes d}$ and $u \in \U$, we have 
$\D(u x)=\varrho(\sigma_\io'(u))\D(x)$. 
\end{lemma}

\begin{proof}
The formula in case of $d=1$ can be verified directly on $u$ being generators and $x=v_a$. 
The formula in general follows by induction on $d$ by noting by Lemma~\ref{lem:DD} that $\varrho$ and $\sigma_\io'$ are coalgebra morphisms. 
\end{proof}

\begin{theorem} 
\label{thm:bisym}
The bilinear form $\la\cdot,\cdot\ra$ on $\V^{\otimes d}$ given in \eqref{eq:BF} is symmetric.
\end{theorem}

\begin{proof}
For $d=1$, by definition \eqref{eq:BF} and using the formulas $\iba(v_a)$ in Lemma~\ref{lem:a}, we compute that $\langle v_a, v_{-a} \rangle =1$, for all $a \in \I$; $\langle v_a, v_{a} \rangle =1$, for all $a \in \Iwr$; and otherwise $\langle v_a, v_{b} \rangle =0$.	Therefore, $\langle \cdot, \cdot \rangle$ is symmetric on $\V$. 

We proceed by induction on $d$. Given $f, g \in \Ibw^d$, write $f'=(f(1),\cdots ,f(d-1)),\  f''=(f(d))$ and similarly for $g', g''$. Hence $M_g =M_{g'}\otimes M_{g''}$. We use $\overline{\phantom{x}}$ to denote $\psi$ and $\overline{\phantom{x}}^\io$ to denote $\iba$ below. The bar map $\iba$ on a tensor product $\U$-module such as $\V^{\otimes d}$ can be defined inductively via $\Theta^\io$ as (cf. \cite[(3.17), Remark 3.14]{BW18a}) 
\begin{align}
  \label{eq:ibarT}
\iba (M_g) = \Theta^\imath(\overline {M_{g'}}^\io\otimes \overline{M_{g''}}). 
\end{align}
Denote $\Theta^\io=\sum a'\otimes a''$ with $a' \in \Ui, a''\in \U$. Then we have 
\begin{align}  
   \label{eq:MMa}
\la M_f,M_g\ra
=&\left(M_{f'}\otimes M_{f''},\D (\Theta^\imath(\overline {M_{g'}}^\io\otimes \overline{M_{g''}}))\right)\\
=&\sum \left(M_{f'},\D(a'\overline{M_{g'}}^\io)\right)\left(M_{f''},\D(a''\overline{M_{g''}})\right).
\notag 
\end{align} 

By Lemma \ref{bisym-lemma2} and the adjunction formula \eqref{bisym-equa}, we have
\begin{align*}
\left(M_{f'},\D(a'\overline{M_{g'}}^\io)\right)
& =\left(M_{f'},\varrho \sigma_\io' (a')\D(\overline{M_{g'}}^\io)\right)
\\
& = \left( \sigma_\imath' (a')M_{f'},\D(\overline{M_{g'}}^\io)\right)
\\
&= \la \sigma_\imath' (a')M_{f'},M_{g'}\ra,
\end{align*}
which, thanks to the symmetry of $\la \cdot, \cdot \ra$ on $\V^{\otimes d-1}$ by the inductive assumption and Proposition~ \ref{prop:ioba}, is equal to
\begin{align}
  \label{eq:MMa1}
\left(M_{f'},\D(a'\overline{M_{g'}}^\io)\right)
& =\la M_{g'}, \sigma_\imath' (a')M_{f'}\ra
= \big( M_{g'}, \D \circ \iba \sigma_\imath' (a') ( \overline{M_{f'}}^\io) \big).
\end{align}
Similarly, we have
\begin{align}
  \label{eq:MMa2}
\left(M_{f''},\D(a''\overline{M_{g''}})\right)
= \big( M_{g''}, \D \circ \sigma\tau (a'') ( \overline{M_{f''}}) \big).
\end{align}
The formula \eqref{eq:MMa2} on $\V$ can be verified directly by definitions for $a''$ being generators of $\U$. (Such a formula is valid in general on $\V^{\otimes d}$; cf. \cite[Proposition 3.3]{CL16} and its proof.) 

Plugging \eqref{eq:MMa1}--\eqref{eq:MMa2} into \eqref{eq:MMa}, we obtain
\begin{align*}
\la M_f,M_g\ra
&= \sum \big( M_{g'}, \D \circ \iba \sigma_\imath' (a') (\overline{M_{f'}}^\io) \big) 
\big( M_{g''}, \D \circ \sigma\tau (a'') ( \overline{M_{f''}}) \big)
\\
&= \Big(M_{g'}\otimes M_{g''}, \D \sum (\iba \sigma_\imath' (a')\otimes \sigma\tau (a'')) (\overline{M_{f'}}^\io \otimes \overline{M_{f''}}) \Big)
\\
&= \Big(M_{g}, \D  (\sigma_\imath \otimes \sigma\tau) (\Theta^\imath) (\overline{M_{f'}}^\io \otimes \overline{M_{f''}}) \Big),
\end{align*} 
which, by Lemma~\ref{lem:iTheta2} and \eqref{eq:ibarT}, can be rewritten as 
\begin{align*}
\la M_f,M_g\ra
&= \big(M_{g}, \D  \Theta^\imath (\overline{M_{f'}}^\io \otimes \overline{M_{f''}}) \big)
\\
&= \big(M_{g}, \D \iba (M_{f'}\otimes M_{f''})   \big)
=\la M_g, M_f \ra.
\end{align*} 
This completes the proof of the theorem. 
\end{proof}

\subsection{An inversion formula}

By Proposition~\ref{prop:iCB} (also see Theorem~\ref{thm:iCBsame}), we can write 
\begin{equation}
 \label{KLP}
C_g=\sum_{y\in \Ibw^d}l_{y,g}(q)M_y, 
\end{equation}
for $l_{y,g}(q) \in \Z[q^{-1}]$; these polynomials $l_{y,g}(q)$ are called {\em (quasi-parabolic) KL polynomials}. Note $l_{g,g}=1$, and $l_{y,g}=0$ unless $y \preceq_\imath g$. 

Similarly, we have 
\begin{equation}
 \label{KLP}
C^*_{g} = \sum_{y\in \Ibw^d} l^*_{y,g}(q)M_{y},
\end{equation}
for $l^*_{y,g}(q) \in \Z[q]$; these polynomials $l^*_{y,g}$ are called {\em (quasi-parabolic) dual KL polynomials}. 
 Note $l^*_{g,g}=1$, and $l^*_{y,g}=0$ unless $y \preceq g$.

\begin{theorem}
 \label{thm:inv}
We have 
$\la C_{g},C^*_{-h}\ra=\delta_{g,h}$, for $g,h\in f\cdot W_d.$
\end{theorem}

\begin{proof}
Since $C_{-h}^*$ is $\iba$-invariant, by \eqref{KLP} we have  
\begin{align}
  \label{eq:gh}
C_g=\sum_{y\in \Ibw^d}l_{y,g}(q)M_y,\quad  
C^*_{-h} = \sum_{-y\in \Ibw^d} l^*_{-y,-h}(q^{-1})\iba(M_{y}).
\end{align}
Similarly, since $C_g$ is $\iba$-invariant, we have
\begin{align}
  \label{eq:hg}
C_{-h}^*=\sum_{y\in \Ibw^d}l^*_{-y,-h}(q)M_{-y},\quad  C_{g} = \sum_{y\in \Ibw^d} l_{y,g}(q^{-1})\iba(M^*_{-y}).
\end{align}
By definition of $\la\cdot,\cdot \ra$ we have 
\[
\la M_y, \iba(M_{-y'}) \ra 
 = (M_y, M_{y'}) =\delta_{y,y'}. 
\]
Therefore, by \eqref{eq:gh} and \eqref{eq:hg} we obtain
\begin{align*}
\la C_g,C_{-h}^*\ra =\sum_y l_{y,g}(q)l^*_{-y,-h}(q^{-1}) &\equiv \delta_{g,h} \pmod{q^{-1}\Z[q^{-1}]}, \\
\la C_{-h}^*,C_g\ra =\sum_y l^*_{-y,-h}(q)l_{y,g}(q^{-1}) &\equiv \delta_{g,h} \pmod{q\Z[q]}.
\end{align*}
By Theorem~ \ref{thm:bisym}, $\la C_g,C_{-h}^*\ra=\la C_{-h}^*,C_g\ra$, and so the above two congruence identities imply that $\la C_g,C_{-h}^*\ra=\delta_{g,h}$.
\end{proof}

We obtain the following inversion formula for quasi-parabolic KL polynomials as a reformulation of Theorem~\ref{thm:inv}; this generalizes \cite{KL79, Do90}. 
\begin{corollary}
  \label{cor:inv}
For all $g,h\in \Ibw^d$, we have
\begin{equation*}
 \sum_{y\in \Ibw^d}l_{y,g}(q)l^*_{-y,-h}(q^{-1})=\delta_{g,h}.
\end{equation*}
\end{corollary}

\begin{remark}
\label{rem:superKL}
The bilinear form $\la \cdot, \cdot \ra$ defined by \eqref{eq:BF} still makes sense for a $\U$-module $\V^{\otimes m} \otimes \V^{* \otimes n}$ as studied in \cite{BW18a}.  
Theorem~\ref{thm:bisym} and a version of Corollary~\ref{cor:inv} remain valid in such a  generality, and it provides an inversion formula for the super Kazhdan-Lusztig polynomials of $\mathfrak{osp}$ type {\em loc. cit.} This generalizes the results in super type A in \cite{CL16}. 
\end{remark}


\end{document}